\documentclass[11pt,reqno, fleqn]{amsart}
\usepackage{amssymb,xcolor,
amsmath,  amsthm, comment,mathrsfs,mathtools
,enumitem
, hyperref, cleveref, 
tikz-cd 
}
\def\cA{\mathcal A} \def\cB{\mathcal B}  \def\cC{\mathcal C} \def\cD{\mathcal D} \def\cE{\mathcal E} \def\cF{\mathcal F} \def\cG{\mathcal G}  \def\cI{\mathcal I}      \def\cO{\mathcal O}  \def\cP{\mathcal P}   \def\cS{\mathcal S}  \def\cT{\mathcal T}  

\def\sF{\mathscr F}

 \def\fc{\mathfrak{c}}
\def\fc{\mathfrak{C}}
  \def\fS{\mathfrak{S}}
 
\def\Z{{\mathbb Z}} \def\R{{\mathbb R}} \def\F{{\mathbb F}}  \def\Q{{\mathbb Q}}

\DeclareMathOperator{\cyc}{cyc}

\DeclareMathOperator{\Prime}{prime}

\renewcommand{\mod}[1]{\,{\rm mod}\,#1}
\renewcommand{\pmod}[1]{\,(\mathrm{mod}\,#1)}
\def\su#1{\sum_{\substack{#1}}}

\def\pr#1{\prod_{\substack{#1}}}
\def\bs#1{\begin{equation*} \begin{split} #1 \end{split} \end{equation*}}
\def\bsc#1{\begin{equation} \begin{split} #1 \end{split} \end{equation}}
\def\eqs#1{\begin{equation*} #1 \end{equation*}}
\def\eqn#1{\begin{equation} #1 \end{equation}}
\def\mult#1{\begin{multline*}#1\end{multline*}}
\def\multn#1{\begin{multline}#1\end{multline}}
\def\ar#1{\left\{ \begin{array}{l@{\quad\text{if }}l} #1 \end{array}\right.
}
\def\fl#1{\left\lfloor#1\right\rfloor} \def\ceil#1{\left\lceil#1\right\rceil}
\def\le{\leqslant} \def\ge{\geqslant}
\def\eps{{\varepsilon}}

\def\set#1{\left\{ #1 \right\}}
\def\abs#1{\left| #1 \right|}
\def\p#1{\left( #1 \right)}

\def\kronecker#1#2{\p{\frac{#1}{#2}}}
\def \ckn {\fc_{n,k}}

\theoremstyle{plain}
\newtheorem{theorem}{Theorem}[section]

\newtheorem{lemma}[theorem]{Lemma}
\newtheorem{proposition}[theorem]{Proposition}

\theoremstyle{definition}

\newtheorem{conjecture}[theorem]{Conjecture}

\theoremstyle{remark}
\newtheorem{remark}[]{Remark}

\title[Average Koblitz Conjecture in Arithmetic Progressions]{Average Twin Prime Conjecture for Elliptic Curves in Arithmetic Progressions}

\author[A. M. G\"ulo\u{g}lu]{Ahmet M. G\"{u}lo\u{g}lu}
\address{Ahmet M. G\"{u}lo\u{g}lu, Department of Mathematics, SA 131, Bilkent University, Ankara, Turkey}
\email{guloglua@fen.bilkent.edu.tr}

\author[A. S. Hamakiotes]{Asimina S. Hamakiotes}
\address{Asimina S. Hamakiotes, Department of Mathematics, Fordham University, Lincoln Center, New York, NY 10023}
\email{ahamakiotes@fordham.edu}

\author[S. M. Lee]{Sung Min Lee}
\address{Sung Min Lee, Department of Mathematics, Wake Forest University, Winston-Salem, NC 27109}
\email{lesum@wfu.edu}

\author[T. Wang]{Tian Wang}
\address{Tian Wang, Max Planck Institute for Mathematics, Vivatsgasse 7, 53111 Bonn, Germany}
\email{tianwmath@gmail.com}


\begin{document}

\begin{abstract}
In 1988, Koblitz conjectured an asymptotic formula for the number of primes $p \le x$ for which the reduction of an elliptic curve over $\mathbb{Q}$ has prime order. Building on the work of Balog, Cojocaru, and David, who proved this conjecture on average in 2011
, we extend the result to primes $p$ lying in arithmetic progressions. Our asymptotic formula holds uniformly for moduli up to a fixed power of $\log x$. We also show that the resulting average constant matches the theoretical constant predicted by Lee, Mayle, and Wang in 2025 using Galois representations.

\end{abstract}


\maketitle


\section{Introduction}

Let $E/\Q$ be an elliptic curve of conductor $N_E$. For a prime $p \nmid N_E$, let $E(\F_p)$ denote the group of $\F_p$-points on the reduction of $E$ modulo $p$. In 1988, motivated by applications to elliptic curve cryptography \cite{MR1945394, 10.1007/978-3-540-24582-7_23}, Koblitz proposed a conjecture on the asymptotic distribution of primes $p$ for which $|E(\F_p)|$ is prime. In particular, he conjectured that unless $E$ is $\mathbb{Q}$-isogenous to a curve with nontrivial torsion, there exists a constant $C_E^{\Prime} > 0$ such that
\eqs{
\pi_E^{\Prime} (x) \coloneqq \#\{p \le x : p \nmid N_E, |E(\F_p)| \text{ is prime}\} \sim C_E^{\Prime} \frac{x}{\log^2 x},
}
as $x \to \infty$. We call $p$ a \textit{Koblitz prime} for $E$ if $|E(\F_p)|$ is prime. 



About two decades later, Jones discovered a counterexample to the conjecture, noting that Koblitz had failed to account for the entanglements of division fields.
Motivated by this, Zywina \cite[Section 1.1]{Zywina2011} refined the constant $C_E^{\Prime}$ to incorporate entanglements, which can vanish even if $E$ is not $\Q$-isogenuous to a curve with non-trivial $\Q$-torsion; in this case, the conjecture predicts that $E$ has only finitely many Koblitz primes.
%


While the Koblitz--Zywina conjecture remains open for any individual curve, several partial results have been established. Notably, Cojocaru \cite{MR2167436} proved a conditional upper bound of the exact order of magnitude for curves without complex multiplication, showing that $\pi_E^{\text{prime}}(x) \ll_E x/\log^2 x$ as $x \to \infty$ under a quasi-Generalized Riemann Hypothesis (quasi-GRH). Conversely, no lower bound of the correct order of magnitude is currently known; however, several conditional and unconditional results analogous to Chen's Theorem have been established, providing lower bounds for the number of primes $p$ for which $|E(\F_p)|$ is an ``almost prime" with at most a prescribed number of prime factors \cite{MR1934487, MR2140162, MR2879973}.


Since the group order $|E(\F_p)|$ is roughly of size $p$ according to the Hasse--Weil bound, the problem is widely recognized as an elliptic curve analogue of the Hardy--Littlewood twin prime conjecture \cite{HardyLittlewood1923, Koblitz}. This analogy provided the foundation for the breakthrough work of Balog, Cojocaru, and David \cite{BCD} in 2011, where they proved  unconditionally that Koblitz's conjecture holds on average over the two-parameter family $\sF$ of elliptic curves given by
\begin{equation}
\label{eq:family}
	\sF = \sF(A,B) \coloneqq \{ E_{a,b} : a,b\in \mathbb{Z},|a| \le A, |b| \le B, \Delta_{E_{a,b}} \neq 0 \}, 
\end{equation}
where $E_{a,b}$ is given by $Y^2 = X^3 + aX + b$ and $A,B > 0$. 

\begin{theorem}[Balog--Cojocaru--David, 2011]\label{thm:BCD}
    Let $x>0$ be a variable and let $\varepsilon >  0$. Let $A=A(x), B=B(x)$ be parameters such that $A,B>x^{\varepsilon}$ and $AB > x\log^{10}x$. 
    Then, 
\eqs{
\frac{1}{|\mathcal{\sF}|} \sum_{E\in \mathcal{\sF}}\pi_E^{\Prime}(x) = C^{\Prime} \frac{x}{\log^2x}+O\left(\frac{x}{\log^3x}\right)  \quad \text{as } x \to \infty,
} 
where $C^{\Prime}$ is given by 
\begin{align*}
    C^{\Prime}  = \prod_{q}\left(1 - \frac{q^2-q-1}{(q-1)^3(q+1)}\right) \approx 0.50517.
\end{align*}
\end{theorem}
Furthermore, Jones \cite{MR2534114} showed that the constants $C_E^{\Prime}$, when averaged over $\sF$, agrees with $C^{\Prime}$. This agreement provides further evidence for  Zywina's refinement of the conjecture. 

More recently, following Zywina's approach, the third and fourth authors together with Mayle proposed \cite[Conjecture 1.6]{Lee_Mayle_Wang_2025} an arithmetic progression analogue of the Koblitz--Zywina conjecture, which we state below. 
For coprime positive integers $n$ and $k$, we define
$$\pi_E^{\Prime}(x;n,k) \coloneqq \#\{p\le x: p \nmid N_E, p \equiv k \pmod n, |E(\F_p)|\text{ is a prime}\}.$$ 

\begin{conjecture}[Lee--Mayle--Wang, 2025] \label{conj:LMW}
If $E/\mathbb{Q}$ is an elliptic curve, then there exists a constant $C_{E,n,k}^{\Prime}\ge 0$, defined in \cite[(35)]{Lee_Mayle_Wang_2025}, for which 
\eqs{
   \pi_E^{\Prime}(x;n,k) \sim C_{E,n,k}^{\Prime} \frac{x}{\log^2 x} \quad \text{as } x \to \infty.
}
\end{conjecture}
In \cite[Corollary 1.10]{Lee_Mayle_Wang_2025}, the authors compute the average of the constants $C^{\Prime}_{E,n,k}$ over the family $\sF$ defined in \eqref{eq:family} and show that
\eqs{
\frac 1 {|\sF|} \sum_{E \in \sF(A, B)} C_{E,n,k}^{\Prime} \rightarrow C_{n,k}^{\Prime},
}
where $A=A(x)$ and $B=B(x)$ are appropriately chosen functions that tend to infinity as $x\to \infty$, and 
\bsc{\label{eq:koblitzconstant}
C_{n, k}^{\Prime} =\frac {1}{\varphi(n)} 
&\pr{q \mid (n,k-1)} \left( 1 - \frac{1}{(q-1)^2 (q+1)} \right) \\
&\cdot \pr{q \mid n\\ q \nmid k-1} \left( 1 - \frac{1}{(q-1)^2} \right) 
\pr{q \nmid n} \Bigl(1 - \frac{q^2 - q - 1}{(q - 1)^3 (q + 1)} \Bigr)
. 
}
In this case, the average constant $C^{\Prime}_{n,k}$ was obtained by analyzing the Galois representations of generic elliptic curves. 



In this paper, following the approach of Balog--Cojocaru--David, we prove that Conjecture~\ref{conj:LMW} holds on average and show that $C^{\Prime}_{n,k}$ indeed reflects the average distribution of Koblitz primes in arithmetic progressions.
\begin{theorem}\label{thm:averagekoblitz}
Suppose that $1 \le k \le n$ with $(k,n)=1$. Let $C, \eps > 0$. Then there exists an $x_0 = x_0 (C,\eps)>0$ such that for $x>x_0$,
\eqs{
	\frac{1}{|\sF|} \su{E \in \sF(A,B)} \pi_E^{\Prime}(x; k, n) = C_{n, k}^{\Prime} \frac{x}{\log^2 x} + O\left( \frac{x}{\varphi(n) \log^3 x}\right),
}
with $C_{n,k}^{\Prime}$ given by \eqref{eq:koblitzconstant}, holds for any $n\ll (\log x)^C$, provided that $A, B > x^\eps$ and  $AB > x\log^{10} x$. 
\end{theorem}

Proving this result presents two main challenges. First, we must track the dependence of the modulus $n$ on $x$ and determine all moduli for which the asymptotic formula holds uniformly. Second, the average constant $\ckn$ initially found using analytic methods is complex, and
we must demonstrate that it decomposes into the conjectured average constant $C_{n, k}^{\Prime}$ given in \eqref{eq:koblitzconstant}.



A standard step in proving Theorem \ref{thm:averagekoblitz} is to 
reduce the global average counting problem to evaluating a corresponding average asymptotic for the local counting functions 
\eqn{\label{nup}
	\nu(p; \cP) = \#\{ E_{s,t} : s,t \in \F_p, \ p\nmid \Delta_{s,t}, \ |E_{s,t}(\F_p)| \in \cP  \}, }
where $\cP$ is the set of primes. 
Consequently, a substantial part of this paper is devoted to establishing the following average result for these local functions: 

\begin{theorem}
\label{thm:local}
Given any $C > 0$, there exists $x_0(C) > 0$ such that for $x > x_0 (C)$, 
\eqn{\label{ALC} 
\su{p \le x\\ p \equiv k \mod n} \nu(p; \cP) = \frac{C_{n,k}^{\Prime}} 3 \frac{x^3}{\log^2 x} + O \Bigl(\frac{x^3}{\varphi(n) \log^3 x}\Bigr)  
}
holds whenever $n \ll (\log x)^C$. 
\end{theorem}

The counting problem above is also related to the arithmetic function
\[
M_E(N;n,k)
:=
\#\{ p \text{ prime}:p\equiv k\mod n, |E(\F_p)|=N \},
\]
which is an arithmetic progression analogue of the function
\[
M_E(N)
:=
\#\{p \text{ prime}: |E(\F_p)|=N\}
\]
introduced by Kowalski \cite{MR2226355}. 
By the Hasse-Weil bound, we see that the asymptotic formula for the sum of $M_E(N;n,k)$ over prime values $N$ is related to the counting function $\pi_E^{\Prime}(x;n,k)$ for the Koblitz conjecture in arithmetic progressions (see also \cite[Conjecture~1.7, p.~818]{MR3248986}).
Thus, our result may also be interpreted as an average asymptotic result for the sum of $M_E(N;n,k)$ over prime values of $N$.
It would be interesting to study the arithmetic properties and average behavior of $M_E(N;n,k)$ in analogy with the results for $M_E(N)$ in e.g., \cite{MR3020306} and \cite{MR3248986}. We leave this as a future project.


\subsection{Koblitz conjecture and related problems}
In this subsection, we explain how the Koblitz conjecture and related problems arise from classical number theoretic questions. 
One of the longstanding problems in number theory is Artin's primitive root conjecture, which asks for the asymptotic behavior of the counting function
$$\pi_a(x)\coloneqq \#\{p \le x : \langle a \pmod p\rangle = \F_p^\times \},$$
given that $a$ is neither a square nor $-1$. An elliptic curve analogue was first considered by Lang and Trotter in \cite{MR427273}. Let $E/\Q$ have positive rank and let $P \in E(\Q)$ be a point of infinite order. One may ask for the asymptotic behavior of
$$\pi_{E,P}(x) \coloneqq \#\{p \le x : \langle P \pmod p \rangle = E(\F_p)\}.$$
There has been partial progress on this problem \cite{CM_1986__58_1_13_0}, but it remains open in general even under the GRH.

In order for $P \pmod p$ to generate $E(\F_p)$, the group $E(\F_p)$ must be cyclic. This leads naturally to the easier question of determining how often $E(\F_p)$ is cyclic. Often referred to as the cyclicity problem, this question was first studied by Serre \cite[pp.465--466]{MR3223094}, who determined the asymptotic under the GRH. Since then, substantial progress has been made by several authors; see, for example, \cite{MR0698163, MR1975393, MR2099195}. In contrast, the corresponding classical question is trivial, since $\F_p^\times$ is cyclic for every prime $p$.

The cyclicity problem has also been studied on average over families of elliptic curves. Banks and Shparlinski \cite{MR2570668} proved that there exists a constant $C^{\cyc}>0$ such that
$$\frac{1}{|\sF|}\sum_{E\in \sF} \pi_E^{\cyc}(x) \sim C^{\cyc}  \frac{x}{\log x},$$
as $x\to \infty$, upon appropriate choices of $A,B$ with respect to $x$. From a complementary perspective, Jones \cite{MR2534114} confirmed that the average of the conjectural density $C^{\cyc}_E$ is equal to the average density $C^{\cyc}$.

One can further consider the arithmetic progression version of the cyclicity problem. The first progress in this direction was made by the first author and Akbal \cite{MR4504664}, who proved under the GRH that there exists a constant $C^{\cyc}_{E,n,k} \ge 0$ such that
\bs{
\pi_E^{\cyc}(x;n,k) &\coloneqq \#\{p\le x : p \nmid N_E, p \equiv k \pmod n, E(\F_p)\text{ is cyclic}\} \\
&\sim C^{\cyc}_{E,n,k}  \frac{x}{\log x},
}
as $x \to \infty$. Subsequent work \cite{MR5065245} determined the cases in which the constant $C^{\cyc}_{E,n,k}$ vanishes, and hence the arithmetic progressions for which there is no prime $p$ such that $E(\F_p)$ is cyclic.

On the average side, using the method of Banks and Shparlinski, the third author \cite{MR4925929} proved that there exists a constant $C^{\cyc}_{n,k}>0$ for which
$$\frac{1}{|\sF|}\sum_{E\in \sF} \pi_E^{\cyc}(x;n,k) \sim C^{\cyc}_{n,k}  \frac{x}{\log x},$$
as $x\to \infty$, upon appropriate choices of $A,B,n$ with respect to $x$. Finally, building on Jones' methods, the third and fourth authors and Mayle  proved in \cite{Lee_Mayle_Wang_2025} that the average of the conjectural density $C^{\cyc}_{E,n,k}$ over $\sF$ is equal to the average density $C^{\cyc}_{n,k}$. 


\subsection{Overview of the paper}\label{subsec:overview}
To prove the geometric problem of counting elliptic curves in $\sF$, we translate it into an analytic problem involving weighted sums over prime pairs. Following the general approach of \cite{BCD}, we carry out this translation in several steps. 


We begin by expressing the average over $\sF$ as a sum over primes of local counting functions $\nu(p; \mathcal{P})$, shifting our focus to a fixed finite field for each prime $p$, counting elliptic curves over $\F_p$ with a prime number of rational points. Carried out in Section \ref{sec:mainthm} and Appendix \ref{app:lem-proof}, this initial step reduces the proof of Theorem \ref{thm:averagekoblitz} to that of \Cref{thm:local}. 
 
Next, in Lemma \ref{lem:first-translate}, we analyze the local counting function by partitioning the curves $E$ according to their Frobenius traces, given by $p + 1 - |E(\F_p)|$.
For each fixed $r$, Deuring's theorem expresses the number of isomorphism classes of elliptic curves over $\F_p$ with trace $r$ in terms of the Hurwitz class number of an associated imaginary quadratic discriminant. We then use the analytic expression for the Hurwitz class number in terms of Dirichlet $L$-functions evaluated at $1$. After truncating the corresponding $L$-series at a suitable parameter and applying quadratic reciprocity, the resulting character sums can be rewritten as explicit congruence conditions on the prime $p$ modulo an appropriate combined modulus. Thus, Lemma \ref{lem:first-translate} converts the original local geometric counting problem into a weighted arithmetic problem involving primes in congruence classes and $r$. 

In Lemma \ref{lem:second-transform}, we reduce the problem to obtaining an asymptotic formula for a weighted sum 
that counts pairs of primes 
with a fixed difference $r$ determined by the trace. This can be viewed as an elliptic curve analogue of the classical Hardy--Littlewood prime pair problem \cite{BCD, HardyLittlewood1923}. The required asymptotic formula is established in Proposition \ref{prop:asym4Flr}, which is later used to complete the proof of Theorem \ref{thm:local} in Section \ref{subsec:proofthmlocal}. Crucially, our proof tracks the modulus $n$ throughout. While all intermediate results leading up to \eqref{eq:choice_of_n} 
hold uniformly for $n$ growing polynomially in $x$, the stronger restriction $n \le (\log x)^{O(1)}$ is invoked in \eqref{eq:choice_of_n} to keep the error term subdominant to the main term in Theorem \ref{thm:local}.

A significant challenge in proving Proposition \ref{prop:asym4Flr}, which establishes the asymptotic formula for the main term, is to uncover the quasi-multiplicative structure of the truncated partial sums. This is where the main differences from the work of \cite{BCD} arise. 
Without the restriction to arithmetic progressions, the relevant arithmetic functions can be analyzed more directly using the theory of multiplicative functions.  Incorporating the arithmetic progression condition,  however, introduces intricate parameter dependencies that obscure this underlying multiplicative structure.

The main term obtained in Proposition \ref{prop:asym4Flr} is evaluated in Section \ref{sec:asym4Flr}, expressing the relevant constant $\ckn$ as an Euler product. One of the main difficulties lies in explicitly computing one of its factors, $\cS(n;k)$, which is done in Section \ref{sec:fullEuler-product}
by combining the evaluation of complete quadratic character sums and Hensel's lemma. 
The error analysis of Proposition \ref{prop:asym4Flr} consists of two principal components. The first component is Proposition \ref{prop:error-term-varE}, which provides a variation of the Barban--Davenport--Halberstam type estimate given in \cite[Theorem 4]{BCD}.
The second component is the remaining errors arising from the truncations in Proposition \ref{prop:fixedr} and \eqref{Crsumfinal}, which use bounds on arithmetic functions together with Rankin's trick.

\subsection{List of notations}

    \begin{enumerate}
    \item We use $n$ to denote a modulus and $k$ an integer coprime to $n$.
    \item We use $p, p'$, $\ell$, and $q$ to denote primes. 
    \item We use $\tau(n)$ to denote the number of distinct positive factors of $n$, and $\varphi(n)$ the Euler's totient function. We denote by $\chi_d(m)$ the Kronecker symbol $\left(\frac{d}{m}\right)$. The  definitions of quadratic characters $\psi_r(q)$ and $\eta_r(q)$ are given in \eqref{psir} and \eqref{etarq}.
    \item We use $\lceil x\rceil$ to denote the ceiling function, and  $\fl{x}$ to denote the floor function.
    \item We use $\log^n x$ to denote $(\log x)^n$.

        \item  
        For $a,b \in \Z$, we denote by $E_{a,b}$ the elliptic curve defined by
        \begin{equation*}
            Y^2=X^3+aX+b,
        \end{equation*}
    and by $\Delta_{a,b}$ its discriminant. For a prime $p\nmid\Delta_{a,b}$, we also denote by $E_{a,b}/\F_p$ its reduction modulo $p$. More generally, for $a,b\in\F_p$ with $\Delta_{a,b}\neq 0$, we denote by $E_{a,b}$ the elliptic curve over $\F_p$ defined by the same Weierstrass equation.

    \item For a prime $p$, we define $\nu(p;\cP)$ as in \eqref{nup}. 
        \item Parameters are introduced in Section \ref{sec:Loc2PrimeSums} and henceforth: 
        \newline
        Let $M \ge 1$ be an integer and $L \coloneqq \fl{\log^M x}$ and $U \coloneqq x^{1/2} L$. We write $Y = x/L$. For $0 \le l \le L-1$, we write $X_l = lY$ and $\cI_l = (X_l,X_{l+1}]$. We take $r$ and $f$ as positive odd integers.
        
\item The following arithmetic functions are used throughout the paper: 
\begin{center}
\begin{tabular}{l|c}
Function & Defined in equation \\
\hline
$\mathcal{F}_l(r)$ & \eqref{Fl(r)} \\
$\mathcal{T}_l(r)$ & \eqref{T(r)} \\
$c_f^r(m)$ & \eqref{cfr(m)} \\
$H(m,f,\beta, \gamma;F)$ & \eqref{eq:definition-H} \\
$Q_r(q)$ & \eqref{Qr} \\
$\mathcal{S}(n, k)$ & \eqref{def:Snk} \\
\end{tabular}
\end{center}

\end{enumerate}

\subsection{Tool and computational resource disclosure}

In accordance with the recommendations of the Leiden Declaration on Artificial Intelligence and Mathematics, we disclose the following uses of large language models in preparing this paper:
we used ChatGPT 5.6 Sol and Gemini 3.6 Flash to check grammar and to improve the precision and clarity of existing sentences and explanations.
Each author individually reviewed the suggestions provided by these models, and evaluated them jointly, incorporating them where we judged them to be correct and appropriate. We have reviewed the final manuscript and take full responsibility for its content. 

\section{Proof of Theorem \ref{thm:averagekoblitz}} \label{sec:mainthm}


The following result  relates the average of $\pi_E^{\Prime}(x; k, n)$ with a weighted average of $\nu(p; \mathcal{{P}})$ for $p\equiv k\pmod n$.


\begin{lemma}\label{lem:average_estimate}
With the conditions stated in Theorem \ref{thm:averagekoblitz}, 
\eqs{
\frac{1}{|\sF|} \su{E \in \sF} \pi_E^{\Prime}(x; k, n) = \su{p \le x\\ p \equiv k \mod n} \frac {\nu(p; \cP)}{(p-1)^2} + O \Bigl( \frac{x}{\varphi(n) \log^3 x} \Bigr),
}
where $\nu(p; \cP)$ is defined in \eqref{nup}.
\end{lemma}

The proof of Lemma \ref{lem:average_estimate} is given in Appendix \ref{app:lem-proof} and closely follows the arguments of \cite{BCD}, with the primary difference being the restriction of the primes to the arithmetic progression $p \equiv k \pmod n$. 


Assuming Theorem \ref{thm:local}, which will be proved in Section \ref{subsec:proofthmlocal}, we can now proceed to the proof of Theorem \ref{thm:averagekoblitz}.
\begin{proof}[Proof of Theorem \ref{thm:averagekoblitz}]
Assume that $n \ll (\log x)^C$ for some positive constant $C$. Let $z=x/\log^3 x$. Since $\nu(p;\cP) \ll p^2$, we see that
\eqs{
\su{  p \le z\\ p \equiv k \mod n} \frac{\nu(p; \cP)}{(p-1)^2} \ll z/n + O(1) \ll  \frac{ x}{\varphi(n) \log^3 x}.
}
For some constant $x_1 > 0$ and $x > x_1$,   
\eqs{
n \ll (\log t)^{2C} \quad \text{for any $t \in (z, x]$.}
}
Thus, there exists $x_2 = x_2(2C)$ for which \eqref{ALC} holds for each $t \in (z,x]$, as long as $x$ is sufficiently large (depending on $x_2$). Therefore, one can apply Theorem \ref{thm:local} and use partial integration to get

\bs{
\su{ z < p \le x\\ p \equiv k \mod n} \frac{\nu(p; \cP)}{(p-1)^2} 
& = \frac{C^{\Prime}_{n,k}} 3 \frac{x}{\log^2 x }  + \frac{2C^{\Prime}_{n,k}} 3 \frac{x}{\log^2 x} \Bigl(1 + O (1/\log x)\Bigr) \\& \qquad + O \Bigl(\frac{ x}{\varphi(n) \log^3 x}\Bigr),\\
&= C^{\Prime}_{n,k} \frac{x}{\log^2 x } + O \Bigl(\frac{ x}{\varphi(n) \log^3 x}\Bigr).
}
The result follows by combining the above findings with Lemma \ref{lem:average_estimate}.

\end{proof}

\section{Transition from Local Counts to Prime Pair Sums}\label{sec:Loc2PrimeSums}

In this section, we translate the problem from local counts to prime pair sums. 
First, we use Deuring's theorem and Dirichlet $L$-functions to convert the local geometric count $\nu(p; \cP)$ into a weighted arithmetic sum of primes in congruence classes to get Lemma \ref{lem:first-translate}. Next, we reformulate this sum as an elliptic analogue of the classical Hardy–Littlewood prime pair problem to get Lemma \ref{lem:second-transform}.


\begin{lemma}\label{lem:first-translate}
Let $x\ge 3$ be a real number, $M \ge 1$ a fixed integer, $L= \fl{\log^M x}$, and $U = x^{1/2} L$. For  $l=0, \ldots, L-1$, we define $Y:=x/L$, $X_l: = lY$, and $\cI_l = (X_l, X_{l+1}]$. Then, there exists an $x_0 > 0$ such that for $x > x_0$ and $n \le Y/2$,  
\multn{\label{eq:first-translate} 
\su{p \le x\\ p \equiv k \mod n} \nu(p; \cP)
= 
\frac 1 {2\pi}
\sum_{1 \le l \le L-1} 
\su{|r|\le 2\sqrt {X_l}} 
\su{p \in \cI_l \\ p \equiv k \mod n \\ p+1 -r \in \cP}  (p-1)\sqrt{4p-r^2} 
\\
\cdot \su{f \ge 1 \\f^2 \mid r^2-4p \\ d \equiv 0,1 \mod 4}  \frac 1 f \sum_{m \le U} \frac  {\chi_d (m)} m +  O\Bigl(\frac{x^3 \log^2 x}{\varphi(n) L} \Bigr),
}
where $d:=(r^2-4p)/f^2$ and $\chi_d(m):= \bigl( \frac d m \bigr)$ is the Kronecker symbol.
\end{lemma}
\begin{proof}
By \cite[Theorem 14.18]{Cox}, we can express $\nu(p;\cP)$ in terms of the Hurwitz class numbers, namely,
\eqs{
\nu(p; \cP) = \su{|r| < 2\sqrt p\\ p+1 -r \in \cP} \frac{p-1}{2} H(r^2 - 4p).
}
For each $r$ with $|r|<2\sqrt{p}$,  the class number is given by

\eqn{\label{Hurwitz}
H(r^2 - 4p) = H(\cO) = \sum_{\cO \subseteq \cO' \subseteq \cO_K} \frac {2h(\cO')}{|\cO'^\times|},
}
where $\cO$ is the order of discriminant $r^2-4p$ in the imaginary quadratic field 

\noindent $K = \Q(\sqrt{r^2 - 4p})$, the sum runs over orders $\cO'$ of $K$ containing $\cO$, and $h(\cO')$ is the class number of $\cO'$. Therefore, 
\eqn{\label{nuwithHurwitz}
\su{p \le x\\ p \equiv k \mod n} \nu(p; \cP) = \su{p \le x\\ p \equiv k \mod n} \su{|r| < 2\sqrt p\\ p+1 -r \in \cP} \frac{p-1}{2} H(r^2 - 4p).
}
Using the following bound 
\eqn{ \label{Hurwitzbound}
H(r^2 - 4p) \ll \sqrt{4p-r^2} \log^2 (4p-r^2) }
from \cite[Lemma 11]{BCD} together with the Brun--Titchmarsh inequality
\eqs{
\pi(x;k,n) \ll \frac{x}{\varphi(n) \log (x/n)},
}
we see that the contribution of the first interval $\cI_0$ is 
\eqs{
\su{p\le Y\\ p \equiv k \mod n} \su{|r| < 2\sqrt p\\ p+1 -r \in \cP} \frac{p-1}{2} H(r^2 - 4p) 
\ll \log^2 x\su{p\le Y\\ p \equiv k \mod n} p^2  
\ll \frac{x^3 \log^2 x}{\varphi(n) L}.
}
Using \eqref{Hurwitzbound} again, we see that the contributions of all $1 \le l \le L-1$ and $p\in \mathcal{I}_l$ in the ranges $2\sqrt{X_l} < |r| < 2\sqrt p$ are 
\eqs{
\ll \frac{x^3 \log^2 x}{n L},
}
where the prime sum is bounded by $O (X_{l+1} Y/n) $. Next, we replace the class number with another expression. If $\cO'$ is an order such that $\cO \subseteq \cO' \subseteq \cO_K$ and $f=[\cO':\cO]$, then $f^2\mid \operatorname{disc}(\cO)$ and 
\[
\operatorname{disc}(\cO')=\frac{\operatorname{disc}(\cO)}{f^2} = \frac{r^2-4p}{f^2} := d.
\]
Thus, each order $\cO' \supseteq \cO$ is uniquely determined by the index $[\cO':\cO]$.
Using the formula 
\eqs{\frac{h(\cO')}{|\cO'^\times|} = \frac{\sqrt{|\operatorname{disc}(\cO')|}}{2\pi} L(1, \chi_{\operatorname{disc}(\cO')}),}
where  $\chi_{\operatorname{disc}(\cO')}$ is the Kronecker symbol, in \eqref{Hurwitz} yields
\eqs{
H(r^2-4p) = \frac{\sqrt{4p-r^2}}{\pi } \su{f \ge 1 \\f^2 \mid r^2-4p \\ d \equiv 0,1 \mod 4}  \frac  {L (1, \chi_d)} f.
}
Taking account of the error terms and using the above expression for the Hurwitz class number, \eqref{nuwithHurwitz} turns into

\mult{
\su{p \le x\\ p \equiv k \mod n} \nu(p; \cP)
= 
\frac 1 {2\pi}
\sum_{1 \le l \le L-1} 
\su{|r|\le 2\sqrt {X_l}} 
\su{p \in \cI_l \\ p \equiv k \mod n \\ p+1 -r \in \cP}  (p-1)\sqrt{4p-r^2} 
\\
\cdot \su{f \ge 1 \\f^2 \mid r^2-4p \\ d \equiv 0,1 \mod 4}  \frac  {L (1, \chi_d)} f  +  O\Bigl(\frac{x^3 \log^2 x}{\varphi(n) L} \Bigr).
}
Using the Polya--Vinogradov inequality, we find that
\eqs{
L (1, \chi_d) = \sum_{m \le U} \frac{\chi_d (m)}{m} + O \Bigl(\frac{\sqrt{|d|} \log |d|}{U}\Bigr).
}
Inserting this into the previous equation yields an error 
\eqs{
\ll \frac{x^{7/2} \log x}{nU} \le \frac{x^3 \log^2 x}{\varphi(n) L}. 
}
Combining everything, we get \eqref{eq:first-translate}.
\end{proof}

We now make the following observations for the indices of the main term in Lemma \ref{lem:first-translate}:
\begin{enumerate}[label=(\roman*)]

\item For $p > 5$, by the Hasse-Weil bound, $p+1-r$ is an odd prime  only if $r$ is odd. We take $x$ large enough so that $Y > 5$. Note that the sum of $\nu(p;\cP)$ over $\cI_0$ is absorbed into the error in Lemma \ref{lem:first-translate}.

\item Since $r$ is odd, $df^2 = r^2 - 4p \equiv 1 \mod 4$. Hence, $f$ is also odd and $d \equiv 1 \mod 4$. 

    \item Since $p \equiv k \mod n$ and $4p \equiv r^2 \mod f^2$, we must have $4k \equiv r^2 \mod (n,f^2)$. 

    \item The character $\chi_d (m)$ can be thought of as a character modulo $4m$. In particular, if $d \equiv a \bmod {4m}$, then $a \equiv 1 \pmod {4}$ and  
\eqs{
\Bigl( \frac d m \Bigr) = \Bigl(\frac a m \Bigr), \qquad p \equiv  \frac{r^2 - af^2}{4} \mod m f^2.
}

\item Since $p \equiv k \mod n$, we must have that
\eqs{
	(n, m f^2) \mid k - \frac{r^2 - af^2} 4.
}
\end{enumerate}

Taking these observations into account and switching the $f$-sum with the $p$-sum, we can update the sums on the right-hand side of \eqref{eq:first-translate} as follows:

\multn{\label{eq:second-translate}
\frac 1 {2\pi}
\sum_{1 \le l \le L-1} 
\su{|r|\le 2\sqrt {X_l}\\ 2 \nmid r} 
\su{f \le 2\sqrt{X_{l+1}} \\ 2 \nmid f\\ (n,f^2) \mid 4k-r^2} 
\frac 1 f \sum_{m \le U} \frac 1 m 
\su{a \mod 4m\\ a \equiv 1 \mod 4\\ (n, m f^2) \mid k - \frac{r^2 - af^2} 4} \Bigl( \frac a m \Bigr) 
\\
\su{p \in \cI_l, p+1 -r \in \cP \\ p \equiv k \mod n \\ p \equiv  (r^2 - af^2)/4 \mod m f^2 \\ }  (p-1)\sqrt{4p-r^2}.
}
Next, we change the weight of the prime sum and the range of the $f$-sum to obtain the following result. 

\begin{lemma}\label{lem:second-transform}
Using the same notation and assumptions as in Lemma \ref{lem:first-translate}, we have the following:
\multn{\label{eq:second-transform}
\su{p \le x\\ p \equiv k \mod n} \nu(p; \cP) 
= \frac 1 {2\pi} \sum_{1\le l\le L-1} \frac{X_l}{\log^2 (X_{l+1})} \su{|r| \le 2\sqrt{X_l}} \cF_l(r) \sqrt{4X_l - r^2}  \\
+ O \Bigl( \frac{x^3}{L} \Bigl( \frac{\log^2 x}{\varphi(n)} + \log x \Bigr) \Bigr),}
where for odd $r\neq 1$, 
\multn{\label{Fl(r)}
\cF_l(r):= \su{m \le U\\f \le L \\ 2 \nmid f, (n,f^2) \mid 4k - r^2} \frac 1 {mf} \su{a \mod 4m\\ a = 1 \mod 4\\ (n, m f^2) \mid k - \frac{r^2 - af^2} 4} \Bigl(\frac a m \Bigr) \\
\cdot \su{
		p \in \cI_l, p+1-r \in \cP \\
		p \equiv k \mod n\\
		p \equiv (r^2 -  af^2)/4 \mod m f^2
}  \log p \log (p+1-r),
}
and $\cF_l(r)=0$ otherwise. 
\end{lemma}

\begin{proof}
For each  $p \in \cI_l$ and $|r| \le 2 \sqrt{X_l}$, 
\mult{
(p-1)\sqrt{4p-r^2} - \frac{X_l \sqrt{4X_l - r^2}}{\log^2 (X_{l+1})} \log p \log (p+1-r) \\
 \ll  Y \Bigl( p^{1/2} + \frac p {\sqrt{4z_l-r^2}} \Bigr),
}
where $z_l = z_l(p) \in (X_l, p)$.
Introducing this weight change in \eqref{eq:second-translate}  yields an error 
\bs{
&\ll Y\sum_{1\le l\le L-1} \su{f\le 2\sqrt{X_{l+1}} \\ m\le U} \frac 1 f \max_{b \mod mf^2} \su{p\in \cI_l \\p\equiv b \mod {mf^2}} \sum_{r\le 2\sqrt{X_l}} \Bigl(p^{1/2} + \frac{p}{\sqrt{4z_l-r^2}}\Bigr)\\
&\ll \frac{x^3 \log U}{L} + x^{5/2} L \log x \ll \frac{x^3 \log x}{L},
} 
for sufficiently large $x$, where we used the fact that $\sum_r (4z_l (p) - r^2)^{-1/2} \ll 1$ uniformly for any $z_l(p)$. Removing the range $L < f \le 2\sqrt{X_{l+1}}$ comes at the cost of an error 
\eqs{
\ll  x^2U L \log x    + \frac{x^3 \log U}{L^2} \ll  \frac{x^3 \log x}{L^2}
}
 for $x$ large enough. Finally, a crude estimate for \eqref{eq:second-translate}  gives an error
\eqs{
\ll \sum_{1\le l\le L-1} X_l^{3/2} \su{m \le U\\f \le L \\ 2 \nmid f, (n,f^2) \mid 4k - r^2} \frac 1 f ( Y/mf^2 + 1 ) 
\ll x^{5/2} \log x 
}
for $r=1$. Combining these error terms with the one in \eqref{eq:first-translate} yields the desired result.
\end{proof}

\section{Dealing with Prime Pair Sums}\label{subsec:weightedprime}

In this section, we present an asymptotic formula in Proposition \ref{prop:asym4Flr} for the average of $\cF_l(r)$, defined in \eqref{Fl(r)}, that is crucial for the proof of Theorem \ref{thm:local}. 
We begin by introducing the weighted prime pair counting function
\eqs{
\Psi (X, Y;r, q, b) =
\su
{
	X < p \le X+ Y \\ 
	p-p'=r \\
	p \equiv b \mod q\\
} 
\log p \log p'.
}
With this definition, $\cF_l(r)$ can be expressed as
\eqs{
\cF_l(r) = \su{m \le U, f \le L\\(n,f^2) \mid 4k - r^2\\2 \nmid f\\  a \mod 4m\\a \equiv 1 \pmod 4\\ (n, m f^2) \mid k - (r^2 - af^2)/4} \frac{1}{mf}
 \Bigl(\frac a m \Bigr) \Psi (X_l, Y;r-1, [ m f^2,n], b),
}
where $b=b(r,a,f,k,n)$ is the unique residue class $b$ modulo $[mf^2,n]$ satisfying the following conditions:
\eqn{\label{resclass}
b \equiv k \mod n, \qquad b \equiv \frac{r^2 - af^2}{4}  \mod m f^2.
}
First, we replace the prime sum with $Y\fS(r-1, [m f^2,n], b)$, where the prime pair constant over a progression (cf. \cite[eq. 34]{BCD}) is given by 
\eqn{\label{twinconstant}
	\fS(r,N,b) = 
	\begin{cases}
		\displaystyle	\frac 2 {\varphi (N)} \prod_{q \neq 2} \dfrac{q(q-2)}{(q-1)^2} \pr{q \mid rN\\ q \neq 2} \dfrac{q-1}{q-2} & \text{if } \begin{array}{l}
		      2 \mid r \text{ and}\\
            (b,N) = 1\\
            (b-r,N) = 1,
            \end{array} 
        \\
		0 & \text{ otherwise}.
	\end{cases}
}
Let $E (X_l, Y;r-1, [ m f^2,n], b)$ denote the difference  
\eqs{\Psi (X_l, Y;r-1, [ m f^2,n], b) - Y\fS(r-1, [m f^2,n], b).}
Recalling the definition of \eqref{Fl(r)}, we can now write 
\begin{equation}\label{eq:Flr}
\cF_l (r)= Y\cT_l (r)+ \cE_l(r),
\end{equation}
where
\eqn{\label{T(r)} 
\cT_l (r) 
 = \su{m \le U, f \le L\\2 \nmid f, (n,f^2) \mid 4k - r^2\\a \mod 4m, a \equiv 1 \pmod 4\\ (n, m f^2) \mid k - (r^2 - af^2)/4} 
	\frac{1}{mf} \Bigl(\frac a m \Bigr) \fS (r-1, [m f^2,n], b),
}
\eqs{
\cE_l (r) 
 = \su{m \le U, f \le L\\2 \nmid f, (n,f^2) \mid 4k - r^2\\a \mod 4m, a \equiv 1 \pmod 4\\ (n, m f^2) \mid k - (r^2 - af^2)/4} 
\frac{1}{mf} \Bigl(\frac a m \Bigr) E (X_l, Y;r-1, [ m f^2,n], b).
}

The following result is a variation of Barban--Davenport--Halberstam type estimate of \cite[Theorem~4]{BCD} with an extra arithmetic progression condition. 


\begin{proposition}\label{prop:error-term-varE} Recall that $L=
\lfloor \log^M x\rfloor$. 
Given $\eps > 0$, let $R$ denote a real number satisfying $x^{1/3+\eps} \le R \le x$. Assume that $n \ll x^C$ for some $C < 1/2$. For any $M'>1$, there exists a constant $x_0 (M', M, C)>0$ such that for all $x\ge x_0 (M', M, C)$,   
\eqs{
\su{|r| \le R\\ 2 \nmid r, r\neq 1} \cE_l(r) \ll_{C, M'} \frac{\tau(n)^{1/2} R x}{(\log x)^{M'-1}}.
}
\end{proposition}
\begin{proof}
Applying the Cauchy--Schwarz inequality successively to the sums over $r$, $m$, and $a$, we obtain 
\mult{
\su{|r| \le R\\ 2 \nmid r, r\neq 1} \cE_l(r)  
\ll 
R^{1/2}\su{f \le L\\2 \nmid f} \frac{1}{f}
\Biggl(\su{|r| \le R\\ 2 \nmid r, r\neq 1\\ (n,f^2) \mid 4k - r^2}
\su{m \le U} \frac 1 m 
\\ \cdot
\su{a \mod 4m\\ a \equiv 1 \pmod 4\\ (n, m f^2) \mid k - (r^2 - af^2)/4} 
\big| E (X_l, Y;r-1, [ m f^2,n], b)\big|^2 \Biggr)^{1/2}.
}
For each $d \mid n$, we group all $m \le U$ with $(m,n)=d$. Then, each $m \le U$ with $(m,n)=d$ gives rise to a different modulus $[mf^2,n]$, and thus to a different error term $E (X_l, Y;r-1, [ m f^2,n], b)$. To see this, note that for $m_1, m_2 \le U$ with $(m_1,n)=(m_2, n)=d$, the equation 
\eqs{
[m_1 f^2, n] = \frac{m_1 f^2 n}{(m_1 f^2, n)} = \frac{m_2 f^2 n}{(m_2 f^2, n)} = [m_2 f^2, n]
}
holds if and only if
\eqs{
\frac{m_1 f^2 n}{d (f^2, n/d)} = \frac{m_2 f^2 n}{d(f^2, n/d)} \iff m_1 = m_2.
}
Hence, for each fixed $f, r, d\mid n$, and $m$ with $(m,n)=d$, 
the map 
\[
a \pmod {4m} \longmapsto b(r, a, f, k, n) \pmod {[m f^2,n]}
\]
is injective. Therefore, we obtain
\mult{
\su{|r| \le R\\ 2 \nmid r, r\neq 1} \cE_l(r)  
	\ll 
	(R\log U)^{1/2}\su{f \le L\\2 \nmid f} \frac 1 f 
\Biggl(
\sum_{d \mid n} 
\su{|r| \le R\\ 2 \nmid r, r\neq 1\\ (n,f^2) \mid 4k - r^2} \\
\cdot \su{m \le U\\ (m,n)=d} \su{a \mod 4m\\ a \equiv 1 \pmod 4\\ (n, m f^2) \mid k - (r^2 - af^2)/4} 
\big| E (X_l, Y;r-1, [ m f^2,n], b)\big|^2\Biggr)^{1/2}}
\eqs{
\ll (R\log U)^{1/2}\log L\Biggl(\sum_{d \mid n}   \su{0<|r|\le R\\q\le nUL^2\\a\pmod q}\left|E(X_l, Y; r-1, q, a) \right|^2 \Bigg)^{1/2}.
}
There exists a constant $x_0= x_0 (M, N, C)$ such that for  $x> x_0$ and any $N\ge 2$, 
\eqs{
Q := nUL^2\ll x^{1/2+C}(\log x)^{3M}\le \frac{x}{(\log x)^{N}}.
}
We are now in a position where we can apply \cite[Theorem 4]{BCD} with $Q$ defined as above
, which gives us that for all sufficiently large $x$, 
\eqs{
\su{|r| \le R\\ 2 \nmid r, r\neq 1} \cE_l(r)  
\ll  (R\log x)^{1/2}(\log\log x)\biggl(  \frac{\tau(n) R x^2}{(\log x)^{2M'}} \biggr)^{1/2},
}
and the desired result follows.
\end{proof}

We turn to the evaluation of $\cT_l(r)$, where we restrict our attention to the values of $r$ satisfying $(n, k+1-r) = 1$ due to the definition of \eqref{resclass} and \eqref{twinconstant}. 
\begin{proposition} \label{prop:fixedr}
For fixed odd integer $r \neq 1$ such that $(n,  k+1-r ) = 1$,
\eqs{
\cT_l (r) = \cC_r +  O \Bigl( \frac{\tau(n)^2 (\log\log (3n))^2 \log \log (3|r|)}{\varphi (n) L^{0.63}}  \Bigr),
}
where $\cC_r$ is given by \eqref{Crfinal}.
\end{proposition}
The proof of Proposition \ref{prop:fixedr} is given in Section \ref{sec:proof-propfixedr} and follows by combining Propositions \ref{prop:Cr} and \ref{prop:Erx}. 

We will now state the principal result of this section, which is the culmination of our local analysis and the final result required to prove Theorem~\ref{thm:local}.  

\begin{proposition}
 \label{prop:asym4Flr}   
Let $\varepsilon>0$ and $C<1/2$. For any $M' > 1$, there exists $x_0 = x_0 (C, M', M, \eps)>0$ such that for any $x>x_0$, 
\multn{\label{asmy4Flr}
\su{|r| \le R\\ r \neq 1} \cF_l (r) = Y\ckn R  
+  O \biggl( \frac{\tau(n)^{1/2} R x}{(\log x)^{M'-1}} 
+ Y (\log\log (3n))^2 \log^2R 
\\
 + YR \frac{\log\log R \ \tau(n)^2 (\log\log (3n))^2}{\varphi (n) L^{0.63}} \biggr)
}
holds for all $ x^{1/3+ \eps}\le R \le x$ and $n\ll x^C$, where $\ckn$ is defined in  \eqref{def:cnk}.
\end{proposition}

By combining the average error term estimate from Proposition~\ref{prop:error-term-varE} with the asymptotic evaluations of the main terms $\cT_l(r)$ established in Proposition~\ref{prop:fixedr} and subsequently averaged over the interval $|r| \le R$, we obtain a complete asymptotic formula for the sum of the weighted prime pair functions $\cF_l(r)$. The proof of Proposition \ref{prop:asym4Flr} can be found in Section~\ref{sec:asym4Flr}.

\section{Proof of Theorem \ref{thm:local}}\label{subsec:proofthmlocal}
In this section, we prove Theorem \ref{thm:local} using Lemma \ref{lem:second-transform}.
\begin{proof}[Proof of Theorem \ref{thm:local}]
Let $\varepsilon>0$.
We first split the $r$-sum  in \eqref{eq:second-transform} at $Z := x^{1/3+\eps}$. 
Using the definition of $\cF_l (r)$ in \eqref{Fl(r)} and a crude estimate, we obtain: 
\bsc{\label{Flrcrudebound}
&\sum_{1\le l\le L-1} 
\frac{X_l}{\log^2 (X_{l+1})} \su{|r| \le Z} \cF_l(r) \sqrt{4X_l - r^2} 
\\
&\ll \sum_l X_l^{3/2} \sum_{|r| \le Z}
\su{m \le U\\f \le L} \frac 1 f \max_{b \mod mf^2} \su{
		p \in \cI_l \\
		p \equiv b \mod m f^2
} 1 
\ll Zx^{5/2}\log x.
}
Let $X = 2 \sqrt{X_l}$. Using Riemann-Stieltjes integration, we obtain
\eqs{
\su{Z < |r| \le X} \sqrt{X^2 - r^2} \cF_l(r) 
= \int_Z^X  \Bigl(\su{0 < |r| \le R} \cF_l(r) \Bigr) \frac{R}{\sqrt{X^2 - R^2}} dR. 
}
By \eqref{asmy4Flr}, for any $M'>1$,  we have
\multn{\label{largeR}
\su{Z < |r| \le X} \sqrt{X^2 - r^2} \cF_l(r)  = 
Y\ckn \int_Z^X  \frac{R^2}{\sqrt{X^2 - R^2}} dR  
\\
+  O \biggl( 
\frac{\tau(n)^{1/2}  xX_l}{(\log x)^{M'-1}} + Y X_l^{1/2} (\log\log n)^2 \log^2 X_l 
\\
+ Y X_l \frac{\log\log x \ \tau(n)^2 (\log\log n)^2}{\varphi (n) L^{0.63}} \biggr),
}
where we use the fact that $\int_Z^X  \frac{R}{\sqrt{X^2 - R^2}} dR\ll X$.
Note that using the formula

\[
\int_Z^X \frac{R^2}{\sqrt{X^2 - R^2}} \, dR=\frac{\pi}{4} X^2 - \frac{X^2}{2} \arcsin\left(\frac{Z}{X}\right) + \frac{1}{2} Z \sqrt{X^2 - Z^2}
\]
and the
Taylor expansion around  $0$, we get
\bsc{\label{largeRmainterm}
Y\ckn \int_Z^X  \frac{R^2}{\sqrt{X^2 - R^2}} dR  & = Y\ckn \frac{\pi X^2} 4 + O\Bigl( \frac{\ckn YZ^3}{X} \Bigr) \\
&= \pi Y\ckn X_l + O \Big( \frac{\ckn x^{1+3\eps}Y^{1/2}}{ l^{1/2}} \Big).
}


Using \eqref{Flrcrudebound}, we can rewrite \eqref{eq:second-transform} as 
\mult{
\su{p \le x\\ p \equiv k \mod n} \nu(p; \cP) 
= \frac 1 {2\pi} \sum_{1\le l\le L-1} \frac{X_l}{\log^2 (X_{l+1})} \su{Z < |r| \le 2\sqrt{X_l}} \cF_l(r) \sqrt{4X_l - r^2} 
\\
 + O \Bigl( Zx^{5/2}\log x + \frac{x^3}{L} \Bigl( \frac{\log^2 x}{\varphi(n)} + \log x \Bigr) \Bigr).
}
Combining \eqref{largeR} and \eqref{largeRmainterm}, the first term on the right of the equation above becomes
\multn{\label{eq:choice_of_n}
\frac 12 \ckn Y^3 \sum_{1 < l \le L} \frac{(l-1)^2}{\log^2 (Yl)}   
+ O \biggl( \frac{\tau(n)^{1/2}  x^3}{(\log x)^{M'-M+1}} + \ckn x^{5/2+3\eps}
\\
 +  x^{5/2}  (\log\log n)^2 \log^2 x 
+ x^3 \frac{\log\log x \tau(n)^2 (\log\log n)^2}{\varphi (n) L^{0.63} \log^2 x} 
\biggr).
}
By Proposition \ref{prop:Euler-product}, we have $\ckn=2C_{n, k}^{\Prime} \ll 1/\varphi(n)$. Therefore, 
 \eqs{
\frac 12 \ckn Y^3 \sum_{1 < l \le L } \frac {(l-1)^2}  {\log^2 (lY)} 
 = C_{n,k}^{\Prime}\frac{ x^3}{3 \log^2 x} +  O \Bigl(\frac{x^3}{\varphi(n) \log^3 x} \Bigr).
}
Using the assumption that $n\ll (\log x)^C$,  provided that $x, M$ and $M'$ are sufficiently large, we have
\bs{
\su{p \le x\\ p \equiv k \mod n} \nu(p; \cP) =C_{n,k}^{\Prime}\frac{ x^3}{3 \log^2 x} +  O \Bigl(\frac{x^3}{\varphi(n) \log^3 x} \Bigr).
}
This completes the proof of Theorem \ref{thm:local}.

\end{proof}

\section{Proof of Proposition \ref{prop:fixedr}} \label{sec:proof-propfixedr}


In this section, we divide the proof of Proposition \ref{prop:fixedr} into two main parts: first, in Section \ref{subsec:doublesum} we identify the main term $\cC_r$ of $\cT_l(r)$, and then in Section \ref{subsec:error term} we estimate the error term. 

We begin by rewriting $\cT_l (r)$ in \eqref{T(r)} as

\eqs{
\cT_l (r) 
 = \su{m \le U, f \le L\\2 \nmid f, (n,f^2) \mid 4k - r^2} 
	\frac{1}{mf} \fS (r-1, [m f^2,n], b) c_f^r (m),
}
where
\eqn{\label{eq:cfr-initial}
c_f^r (m) :=
\su{
	a \mod 4m\\ a \equiv 1 \pmod 4\\ (n, m f^2) \mid k - (r^2 - af^2)/4 \\
	(b(b-r+1),[mf^2,n])=1
	} \Bigl(\frac a m \Bigr),
}
with $b:=b(r, a, f, k, n)$ defined by \eqref{resclass}. Below the arithmetic properties of $c_f^r(m)$ needed for the analysis of the main term are established. 

\subsection{Properties of $c_f^r$}\label{subsec:propcfr}
The coprimality condition, $(b(b-r+1),[mf^2,n])=1$ on $a$, in \eqref{eq:cfr-initial} can be split into the following conditions:
\eqs{
(n,k(k-r+1)) = 1, \qquad  \Bigl(mf, \frac{r^2 - af^2} 4 \frac{(r-2)^2 - af^2}4 \Bigr) = 1,
}
where the latter implies that $(r,f)= 1 = (r-2,f)=1$. Thus, it can be replaced by the conditions
\eqs{
(m, 4^{-1}(r^2 - af^2)) = 1 = (m, 4^{-1} ((r-2)^2 - af^2)),
}
leading to the new expression

\eqn{\label{cfr(m)}
c_f^r (m) =
\su{
	a \mod 4m\\ a \equiv 1 \pmod 4\\ (n, m f^2) \mid k -  \frac{r^2 - af^2}{4}    \\
	(m, \frac{r^2 - af^2}{4}) = 1\\
    (m, \frac{(r-2)^2 - af^2}{4})= 1
	} \Bigl(\frac a m \Bigr),
}
which holds if $(n,k+1-r)=1$, and is zero otherwise.

The following lemma establishes the multiplicativity of $c_f^r(m)$ in $m$ and determines its local factors. 
\begin{lemma} \label{lem:cfr(m)}
Let $r \neq  1$ be an odd integer such that $(n,k+1-r) = 1$ and let $f$ be a positive odd integer satisfying $(r(r-2),f) = 1$ and $(n,f^2) \mid r^2 - 4k$.  
\begin{enumerate}
\item The function $c_f^r (m)$ is multiplicative in $m$. 

\item If $q \mid f$ and $v_q(n) \le 2v_q(f)$, then
    \eqn
{\label{2f >= n}
c_f^r(q^i) =
\begin{dcases}
    	0 & \text{if } 2 \nmid i, \\
	\varphi(q^i) & \text{if } 2 \mid i.
\end{dcases}
}

\item If $q \mid f$ and $v_q(n) > 2v_q(f)$, then necessarily $2v_q(f) \le v_q(r^2 - 4k)$ and 

\eqn{\label{n > 2f}
c_f^r (q^i) 
= 
\frac{q^{i+2v_q(f)}}{(q^{i+2v_q(f)}, n)} \cdot \begin{cases}
    \psi_r (q) & \text{if } 2v_q(f) = v_q (r^2-4k), \\
    0 & \text{otherwise,}
\end{cases}
}
where
\eqn{\label{psir}
\psi_r (q) = \Bigl( \frac{(r^2-4k)/q^{v_q (r^2-4k)}}{q} \Bigr).
}

\item If $q \nmid f$, then 
\eqs{
c_f^r (q^i) = 
\begin{dcases}
	(-2)^i /2  & \text{if } q=2 \text{ and } 2 \nmid n, \\
    \frac{(\eta_r (q) q)^i}{(n,q^i)} & \text{if } q \mid n,
\end{dcases} 
}
where  
\eqn{\label{etarq}
\eta_r (q) = \Bigl( \frac{r^2-4k}{q} \Bigr).
}

\item If $q \nmid 2nf$, 
\eqs{
\frac{c_f^r(q^i)}{q^{i-1}} =
\begin{dcases}
    q - 2, & \text{if } 2 \mid i, q \mid r(r-2)(r-1),\\
     -1, & \text{if }  2 \nmid i, q \mid r(r-2)(r-1),\\
     q - 3, & \text{if }  2 \mid i, q \nmid r(r-2)(r-1),\\
    -2, & \text{if }  2 \nmid i, q \nmid r(r-2)(r-1).
\end{dcases}		
}
\end{enumerate}
\end{lemma}

\begin{proof}
To see multiplicativity, suppose $(m_1, m_2)=1$ and $(n, m_if^2)\mid (k-(r^2-a_if^2)/4)$. Let $b$ denote the unique class modulo $[4m_1, 4m_2] = 4m_1m_2$ for which $b \equiv a_i \mod 4m_i$ with $a_i = 1 \mod 4$ for $i=1,2$. 
Since $(n,m_if^2)  \mid (a_i f^2 - b f^2)/4$, it follows that $(n,m_i f^2) \mid (k- (r^2 - bf^2)/4)$. Thus,
\eqs{
(n,m_1m_2 f^2) = [(n,m_1 f^2), (n,m_2 f^2)] \mid k- (r^2 - bf^2)/4 .
}
The other conditions in \eqref{cfr(m)} follow from the fact that 
\[
(m_1m_2, \star) = (m_1, \star)(m_2,\star).
\]
This completes the proof of (1).

For $q \nmid n$, since $(n, f^2)\mid r^2-4k$ and $f$ is odd, the condition $(n,qf^2)  \mid (k - (r^2-af^2)/4)$ holds naturally. Therefore, part (5) follows from the evaluation in \cite[Lemma 16 parts (3)-(5)]{BCD}. 

From now on, we assume that $q \mid n$. Given a prime $q$, write $f = q^{v_q(f)} h$, where $q \nmid h$. Since $(n,h^2) \mid (n,f^2)\mid k- (r^2 - af^2)/4$, the first divisibility condition on $a$ is equivalent to 
\eqs{
	(n,q^{i+2v_q(f)}) \mid k- (r^2 - af^2)/4.
}
Furthermore, note that if
\eqs{
	q \mid \frac{r^2 - af^2}{4}\cdot    \frac{(r-2)^2 - af^2}{4},
}
then $q \nmid n$ since otherwise we would get $q \mid (n, k(k-r+1))=1$. Therefore, since $q \mid n$, we can simplify $c_f^r (q^i)$ as

\eqs{
c_f^r (q^i) = 
\su{
	a \mod 4q^i \\
	a \equiv 1 \pmod 4 \\
	(n, q^{i+2v_q(f)})\mid k- (r^2 - af^2)/4  
	} 
	\Bigl(\frac a q \Bigr)^i.
}

If $q \mid f$, then $q \mid (n,f^2) \mid r^2 - 4k$, and hence $q$ is odd. Furthermore, 
\eqs{
\min \{ v_q(n), 2v_q(f) \} \le v_q(r^2 - 4k),
}
and the last condition on $a$ is equivalent to 
\eqs{
(n, q^{i+2v_q(f)}) \mid 4k - r^2 + a q^{2v_q(f)} h^2.
}
If $v_q(n) \le 2v_q(f)$, then the last condition on $a$ trivially holds. Therefore, part (2) follows from \cite[Lemma 16 part (4)]{BCD}.
If $v_q(n) > 2v_q(f)$, then necessarily $2v_q(f) \le v_q(r^2 - 4k)$ and
\eqs{
c_f^r (q^i) 
= 
\su{
		a \mod 4q^i, a \equiv 1 \pmod 4 \\
	(nq^{-2v_q(f)}, q^i) \mid \frac{4k- r^2}{q^{2v_q(f)}}  + ah^2 
	} 
	\Bigl(\frac {ah^2} q \Bigr)^i 
= \Bigl( \frac{\frac{r^2-4k}{q^{2v_q(f)}}}{q} \Bigr)^i \frac{q^{i+2v_q(f)}}{(q^{i+2v_q(f)}, n)} ,
}
which is $0$ unless $2v_q(f) = v_q(r^2 - 4k)$. Hence, part (3) of the lemma follows. 

Finally, if $q \nmid f$ and $q\neq 2$, then 
\eqs{
c_f^r (q^i) 
= \su{
	a \mod 4q^i \\
	a \equiv 1 \pmod 4 \\
	(n, q^i )\mid 4k - r^2 + af^2  
	} 
	\Bigl(\frac {af^2} q \Bigr)^i = 
 \Bigl( \frac{r^2-4k}{q} \Bigr)^i \frac{q^i}{(n,q^i)}.
}
If $2 \mid n$ and $q=2\mid f$, then
\eqs{
c_f^r (2^i) = 
 \su{a \mod 2^{i+2}, a \equiv 1 \pmod 4 \\
	4(n, 2^i)\mid af^2 - (r^2 - 4k)  } \Bigl(\frac {af^2} 2 \Bigr)^i 
= \Bigl( \frac{r^2-4k}{2} \Bigr)^i 
\frac{2^i}{(n,2^i)} =  \frac{(-2)^i}{(n,2^i)}
}
since $2 \nmid k$ implies $r^2 - 4k \equiv 5 \mod 8$. Therefore, part (4) holds. 

This completes the proof of the Lemma.
\end{proof}


\subsection{Factorization of the double sums}\label{subsec:doublesum}

Next, we will give an Euler product representation of the main term $\cC_r$ of $\cT_l(r)$ in Proposition \ref{prop:Cr}. 
To see where the main term comes form, we insert the twin prime constant defined in  \eqref{twinconstant} into \eqref{T(r)} and rearrange terms to find that
\eqn{\label{incompletesum}
\cT_l(r) = \frac 2 {\varphi (n)} \pr{q \mid (r-1)n\\ q \neq 2} \frac{q-1}{q-2} \prod_{q \neq 2} \frac{q(q-2)}{(q-1)^2} 
\su{m \le U\\f \le L\\ (f,2r(r-2))=1\\(n,f^2) \mid 4k - r^2} H(m,f,2,3;c_f^r),
}
where the new condition $(f, 2r(r-2))=1$, implied by the properties of $c_f^r(m)$, has been incorporated into the $f$-sum. For any $\beta, \gamma > 0$, we define 

\multn{\label{eq:definition-H}
H(m,f,\beta, \gamma;F) :=
	\frac { \varphi((mf^2,n))F(m)} 
	{\varphi(m^\beta) \varphi (f^\gamma)   } 
	\prod_{q \mid (m,f)} \frac{q-1} q
	\\
\cdot	\pr{q \mid m \\q \neq 2} \frac{q-1}{q-2}
	\pr{q \mid (m,(r-1)fn) \\q \neq 2} \frac{q-2}{q-1}
	\pr{q \mid f \\q \neq 2} \frac{q-1}{q-2}
	\pr{q \mid (f,(r-1)n) \\q \neq 2} \frac{q-2}{q-1},
}
with $\varphi(m^\beta):=m^\beta \prod_{p\mid m} (1 - 1/p)$. 

The main term is obtained by completing the truncated sums over $m$ and $f$. Namely, we define
\multn{\label{Cr}
\cC_r := \frac 2 {\varphi (n)} \pr{q \mid (r-1)n\\ q \neq 2} \frac{q-1}{q-2} \prod_{q \neq 2} \frac{q(q-2)}{(q-1)^2}  \\
\cdot \su{m, f \ge 1\\ (f,2r(r-2))=1\\(n,f^2) \mid 4k - r^2} H(m,f,2,3;c_f^r).
}
Therefore, $\cC_r$ differs from $\cT_l(r)$ only through the tails introduced by extending the ranges $m\le U$ and $f\le L$ to infinity. 
We write the error as
\eqn{\label{eq:Erx}
E_r(x):=\cT_l(r)-\cC_r.
}
The contribution of $E_r(x)$ will be estimated in the next section. For the remainder of this section, we focus on the arithmetic structure of $\cC_r$ and express it as an Euler product.

To separate the interaction between $m$ and $f$, we fix $f$ and decompose $m$:
\[
m=m's,
\qquad
(m',f)=1,
\qquad
s\mid f^\infty,
\]
where $s\mid f^\infty$ means that every prime divisor of $s$ also divides $f$. Substituting this decomposition into \eqref{Cr}  and using the letter $m$ again instead of $m'$ gives
\eqn{\label{mfsumfactorazion}
\su{m, f \ge 1\\ (f,2r(r-2))=1\\(n,f^2) \mid 4k - r^2} H(m,f,\beta,\gamma; F) = 
\su
{
	f \ge 1\\(2r(r-2),f)=1\\ (n,f^2) \mid 4k - r^2  
}
\frac {M(f,\beta; F) S(f,\beta;F) P(f)} { \varphi (f^\gamma)}, 
}
where 

\bsc{\label{M(f)}
M(f,\beta;F) &=  
	\su{(m,f)=1} \frac { \varphi((m,n)) F(m)} 
	{\varphi(m^\beta)  } 
	\pr{q \mid m \\q \neq 2} \frac{q-1}{q-2}
	\pr{q \mid m\\q \mid (r-1)n \\q \neq 2} \frac{q-2}{q-1},
	\\
P(f) &= \pr{q \mid f} \frac{q-1}{q-2}
\pr{q \mid (f,(r-1)n)} \frac{q-2}{q-1}, \\
S(f,\beta;F) &= \su{s \mid f^\infty }
	\frac { \varphi((sf^2,n)) F(s)} 
	{  s^\beta    }.
}

Since all terms of $M(f,\beta;F)$ are multiplicative in $m$, we can factor it as 
\eqs{ 
M(f,\beta;F) 
= M_2(f,\beta;F) \prod_{q \nmid 2nf} M_q (f,\beta;F) \pr{q \mid n \\ q \nmid 2f} M_q (f,\beta;F), 
}
where for any prime $q$, 
\eqs{
M_q (f,\beta;F) = 1 + \sum_{i \ge 1} \frac { \varphi((q^i,n)) F(q^i)} 
{\varphi(q^{\beta i})  } 
\pr{\ell \mid q \\\ell \neq 2} \frac{\ell-1}{\ell-2}
\pr{\ell \mid (q,(r-1)n) \\\ell \neq 2} \frac{\ell-2}{\ell-1}.
}
Let 
\bsc{\label{A_r(q)}
A_r (q, \beta; i)= 1 +
\begin{dcases}
	\frac {q-3 + (-1)^i 2 q^{ \beta-1 }}{(q-2) (q^{2(\beta-1)} -1)}   & \text{if } q \nmid r(r-1)(r-2),\\
	\frac {q-2 + (-1)^i   q^{ \beta-1 }}{(q-1) (q^{2(\beta-1)} -1)}  
	& \text{if } q \mid r-1, \\
	\frac {q-2 + (-1)^i  q^{ \beta-1 }}{(q-2) (q^{2(\beta-1)} -1)}  
	& \text{if } q \mid r(r-2).
\end{dcases} 
}
Using Lemma \ref{lem:cfr(m)}, we can derive the following via a routine computation:  
\begin{enumerate}
\item[(i)] For $q=2$, 
\eqs{
M_q (f,\beta;F)  =  \begin{dcases}
	 \frac 1 {1 + 2^{1-\beta}}    & \text{if } F = c_f^r,\\
	 \frac 1 {1 - 2^{1-\beta}}    & \text{if } F = |c_f^r|.
\end{dcases}
}
\item[(ii)] For $q \nmid 2nf$, 
\eqs{
M_q (f,\beta;F) = \begin{dcases}
	A_r (q, \beta;1) &\text{if } F = c_f^r, \\
	A_r (q, \beta;2) &\text{if } F = |c_f^r|. 	 
\end{dcases}
}

\item[(iii)] For $q \mid n$ but $q \nmid 2f$, 
\eqs{
M_q (f,\beta;F)  = 
\begin{dcases}
	\Bigl(1- \frac { \eta_r (q) } { q^{ \beta-1  }   }\Bigr)^{-1}    & \text{if } F = c_f^r,\\
	\Bigl(1- \frac { \eta_r (q)^2} { q^{ \beta-1  }   }\Bigr)^{-1}  & \text{if } F = |c_f^r|.
\end{dcases}
}

\end{enumerate}

As a result, it follows that $\prod_{q \nmid 2nf} M_q (f,\beta;F) \ll \zeta(2\beta-2) $ for $\beta \in (3/2,2]$, hence
\eqn{\label{M(f)bound}
M(f, \beta; |c_f^r|) 
\ll  \frac 1 {2\beta-3} \pr{q \mid n} \Bigl(1- \frac 1 { q^{ \beta-1  }   }\Bigr)^{-1}, }
and that
\multn{\label{Mfactored}
M(f,2;c_f^r) 
= \frac 2 3 \prod_{q \nmid 2nf} A_r (q,2;1) \pr{q \mid n \\ q \nmid 2f} \Bigl( 1- \frac { \eta_r (q) } q  \Bigr)^{-1}, \\
= \frac 2 3 \prod_{q \nmid 2n} A_r (q,2;1) 
\pr{q \mid n \\ q \nmid 2} \Bigl( 1- \frac { \eta_r (q) } q  \Bigr)^{-1} 
\pr{q \mid f\\ q \nmid 2n} A_r (q,2;1)^{-1}.
}

\begin{lemma}\label{lem:mult-S-f}
The function $S(f, \beta; c_f^r)$ is multiplicative in $f$.  
\end{lemma}
\begin{proof}
Let  $f=p^{t}$ for a prime $p$, and $g$ be an integer not divisible by $p$. By definition, 
\eqs{
S(fg, \beta; c_{fg}^r) = \sum_{s\mid (fg)^\infty } \frac{\varphi((s(fg)^2, n)) c_{fg}^r(s)}{s^\beta}.
}
By writing $s=s'p^e$ where $s' \mid g^\infty$, we have
\eqs{
c_{fg}^r (p^e s') = c_{fg}^r (p^e) c_{fg}^r (s') = c_f^r (p^e) c_g^r (s').
}
Thus, we get 
\bs{
S(fg, \beta; c_{fg}^r) 
&=\su{e \ge 0, s' \mid g^\infty} \frac{\varphi((p^e f^2, n)) \varphi((s'g^2, n)) c_f^r (p^e) c_g^r (s')}{(p^es')^\beta},
\\
&=S(f, \beta; c_f^r) S(g, \beta; c_g^r). 
}
A similar argument shows that $S(f, \beta; |c_f^r|)$ is multiplicative. 
\end{proof}

\begin{proposition}\label{prop:Cr}
Recall \eqref{etarq} and \eqref{Cr}. Then, we have 
\multn{
\label{Crfinal}
\cC_r = 
\frac {4} {3\varphi (n)}  
\prod_{q \neq 2} \frac{q^2(q^2 - 2q -2)}{(q+1)(q-1)^3} 
\pr{q \mid n \\ q \neq 2}
\frac{(q-1)(q^2-1)}{q(q^2 - 2q -2)}
\pr{q \mid n \\ q \nmid 2r(r-2)}  Q_r (q) 
\\
\cdot 
\pr{q \mid r(r-2)\\ q \nmid 2n} \frac{q^2 - 2 q - 1}{q^2 - 2 q - 2} 
\pr{q \mid  r-1 \\ q \nmid 2n}
\frac{q^2-q-1}{q^2 - 2q -2}
\pr{q \mid n\\q \neq 2} 
\Bigl(1- \frac{ \eta_r (q) } q \Bigr)^{-1},
}
where 
\begin{equation}\label{Qr}
Q_r(q)=\begin{cases}
1 & \text{if } \nu_r =0, \\[2mm]
\dfrac{q(q^{\ceil{\nu_r/2}}-1)}
 {q^{\ceil{\nu_r/2}} (q-1)} + \dfrac{\delta_{2 \mid \nu_r} q }{q^{\nu_r/2} (q - \psi_r (q))}  & \text{if } 
 0 < \nu_r < v, 
 \\[3mm]
\dfrac{q(q^{\ceil{v/2}}-1)}{q^{\ceil{
v/2}} (q-1)} + \dfrac{q^{2+v-3\ceil{v/2}}}{q^2-1}  & \text{if } \nu_r \ge v,
\end{cases}
\end{equation}
with $v:=v_q (n)$ and $\nu_r = v_q (r^2 - 4k)$.

\end{proposition}
\begin{proof}
Inserting \eqref{Mfactored} into 
\eqref{mfsumfactorazion}, we find that
\multn{\label{eq:prodofH}
\su{m, f \ge 1\\ (f,2r(r-2))=1\\(n,f^2) \mid 4k - r^2} H(m,f,2,3; F) = 
\frac 2 3 \prod_{q \nmid 2n} A_r (q,2;1) 
\pr{q \mid n \\ q \nmid 2} \Bigl( 1- \frac { \eta_r (q) } q  \Bigr)^{-1} 
\\
\cdot \su
{
	f \ge 1\\(2r(r-2),f)=1\\ (n,f^2) \mid 4k - r^2  
}
\frac {S(f,2;c_f^r) P(f)}
{\varphi (f^3)}
\pr{q \mid f\\ q \nmid 2n} A_r (q, 2;1)^{-1}, }
where the $f$-sum can be factored using   
Lemma \ref{lem:mult-S-f} as 
\eqs{
\su
{
	f \ge 1\\(2r(r-2),f)=1\\ (n,f^2) \mid 4k - r^2  
}
\frac {S(f,2;c_f^r) P(f)}
{\varphi (f^3)}
\pr{q \mid f\\ q \nmid 2n} A_r (q, 2;1)^{-1} = Q \pr{q \mid n \\ q \nmid 2r(r-2)}  Q_r (q),
}
with 
\bs{
Q&= \pr{q \nmid 2nr(r-2)} \Bigl(
1 + 
A_r (q,2;1)^{-1} \frac{q}{q-1}	\su{j \ge 1} 
\frac{S(q^j,2;c_{q^j}^r) P(q^j)} {q^{3j}}  
\Bigr), \\
Q_r (q) & = 1 + \su{j \ge 1\\ (n,q^{2j}) \mid 4k-r^2 }  \frac 1 {q^{3j}} \su{i \ge 0} 
	\frac { (q^{i+2j},n) c_{q^j}^r(q^i)}{q^{2i}}.
}
For $Q$, recalling the definition of $S(q^j,\beta;F)$ and using \eqref{2f >= n}, we can show that
\eqs{
S(q^j,\beta;F)  = \sum_{i \ge 0} 
	\frac { \varphi((q^{i+2j},n)) F(q^i)}{q^{i\beta }} = 1 + \frac{q-1}{q(q^{2(\beta-1)}-1)}.
}
Combining this with \eqref{A_r(q)}, it follows that
\eqs{
Q   =  
\pr{q \nmid 2nr(r-1)(r-2)} \frac{q(q^2 - 2q -2)}{q^3-2q^2-2q-1}
\pr{q \mid r-1\\ q \nmid 2n
} \frac{q(q^2-q-1)}{q^3-q^2-q-1}.
}
Putting the factored form of \eqref{eq:prodofH} into the definition of $\cC_r$ in \eqref{Cr}, we obtain
\multn{\label{eq:newCr}
\cC_r = \frac {4Q} {3\varphi (n)} \prod_{q \neq 2} \frac{q(q-2)}{(q-1)^2}  \pr{q \mid (r-1)n\\ q \neq 2} \frac{q-1}{q-2}  
\\
\cdot 
\pr{q \mid n \\ q \nmid 2r(r-2)}  Q_r (q) \pr{q \nmid 2n} A_r (q, 2;1) 
\pr{q \mid n\\q \neq 2} 
\Bigl(1- \frac{ \eta_r (q) } q \Bigr)^{-1}.
}
Using \eqref{A_r(q)} once more, we find that
\bs{
Q  \pr{q \nmid 2n} A_r (q, 2;1)  &= 
\pr{q \nmid 2n} \frac{q(q^2 - 2q -2)}{(q-2)(q^2-1)}
\pr{q \mid r(r-2)\\ q \nmid 2n} \frac{q^2 - 2 q - 1}{q^2 - 2 q - 2} \\
& \qquad \cdot \pr{q \mid  r-1 \\ q \nmid 2n}
\frac{(q-2)(q^2-q-1)}{(q - 1)(q^2 - 2q -2)}.
}
By inserting the above into \eqref{eq:newCr} and simplifying, we end up with \eqref{Crfinal}.

As for $Q_r (q)$, if $\nu_r = 0$, then $Q_r (q) = 1$. Otherwise, we assume that $r \equiv u \mod n$ and consider the remaining cases. If $v > \nu_r > 0$, then
\eqs{v_q (r^2 - 4k) = v_q (u^2 - 4k), \qquad \psi_r (q) = \psi_u (q)\neq 0.}
By \eqref{n > 2f}, it follows that 
\bs{ 
Q_r (q) &= \su{0 \le j < \nu_r/2 }  \frac {q^{2j}} {q^{3j}} + \frac {\delta_{2 \mid \nu_r}   }{q^{3\nu_r/2}} \su{i \ge 0} 
\frac { (q^{i+\nu_r},n) }{q^{2i}} \frac{\psi_r (q)^i q^{i+\nu_r}}{(q^{i+\nu_r}, n)},
}
which gives the second case in \eqref{Qr}. Finally, if $v \le \nu_r$, then $v_q(u^2 -4k) \ge v$ and
\bs{ 
Q_r (q) &= \su{0 \le j < v/2  } \frac{q^{2j}} {q^{3j}}  + \su{i \ge 0\\ 2 \mid i} \frac{ q^{v} \varphi(q^i)}{q^{2i}}  \su{j \ge 1\\2j \ge v  }  \frac{1} {q^{3j}},  
}
which yields the last assertion in \eqref{Qr}.
\end{proof}

\begin{remark} 
In both cases above, note that $Q_r (q) = Q_u (q)$ if $r\equiv u \pmod n$.
\end{remark}

\subsection{Estimate of the error term $E_r (x)$}\label{subsec:error term}

We are finally ready to bound the error term $E_r (x)$ in \eqref{eq:Erx}.
Let $E(U, L):=E(U, L; r)$ denote the error incurred in completing the double sums over $m$ and $f$ defined by 
\eqn{
\label{eq:E-U-L}
\Bigg|\su{m, f \ge 1\\ (f,2r(r-2))=1\\(n,f^2) \mid 4k - r^2} H(m,f,2,3;c_f^r) - \su{m \le U\\f \le L\\ (f,2r(r-2))=1\\(n,f^2) \mid 4k - r^2} H(m,f,2,3;c_f^r)\Bigg|,
}


The following result generalizes and refines \cite[Proposition~15]{BCD} by making the dependence of the upper bound on $n$ explicit.  We apply Rankin’s trick (see the proof for more detail) to convert the incomplete sums into complete double sums and then factor these double sums. Although our proof is more involved than a crude estimate obtained using a simpler argument, this way we get a better bound in the modulus $n$. 
\begin{proposition}\label{prop:Erx}
Recall definitions \eqref{eq:definition-H}, \eqref{eq:Erx}, and \eqref{eq:E-U-L}. Then,
\eqn{\label{EULbound}
E(U, L) 
\ll \frac{\tau(n)^2 \log\log n}{L^{0.63}}.
}  
Consequently, 
\eqs{
E_r(x)\ll \frac{\tau(n)^2 \log\log n}{\varphi (n) L^{0.63}} \pr{q \mid (r-1)n\\ q \neq 2} \frac{q-1}{q-2} .
}
\end{proposition}
\begin{proof}

Applying the triangle inequality and using Rankin's trick\footnote{For any sequence $a_n \in \R^+$,
\eqs{
\sum_{n > x} a_n \le \sum_{n \ge 1} (n/x)^\sigma a_n =  \frac{1}{x^\sigma} \sum_{n \ge 1}  n^\sigma a_n. 
}
for any $\sigma\ge 0$ and $x>0$.}, for any $\beta \le 2$ and $\gamma \le 3$, we have  
\multn{\label{rankintrick}
E(U, L) 
=\Bigg| \su{m \ge U, f\le L \\ (f,2r(r-2))=1\\(n,f^2) \mid 4k - r^2} H(m,f,2,3;c_f^r) + \su{m \ge 1, f \ge L \\ (f,2r(r-2))=1\\(n,f^2) \mid 4k - r^2} H(m,f,2,3;c_f^r)\Bigg|\\
\le  \su{m, f \ge 1\\ (f,2r(r-2))=1\\(n,f^2) \mid 4k - r^2} \bigl(U^{\beta-2}  H(m,f,\beta,3; |c_f^r|)\\
+ L^{\gamma-3} H(m,f,2,\gamma; |c_f^r|) \bigr).
}
By \eqref{M(f)bound}, it follows that
\multn{\label{costofcompletion}
	\su{m, f \ge 1\\ (f,2r(r-2))=1\\(n,f^2) \mid 4k - r^2} H(m,f,\beta,\gamma;|c_f^r|) \\
\ll 
	\frac 1 {2\beta-3} \pr{q \mid n} \Bigl(1- \frac 1 { q^{ \beta-1  }   }\Bigr)^{-1}
	\su
	{
		f \ge 1\\(2r(r-2),f)=1\\ (n,f^2) \mid 4k - r^2  
	}
	\frac {S(f,\beta;|c_f^r|) P(f)}
	{\varphi (f^\gamma)}.
}
Note that since all functions  in the $f$-sum in \eqref{costofcompletion} are multiplicative, we can factor the $f$-sum as 

\eqn{\label{eq:F1F2F3}
\su
	{
		f \ge 1\\(2r(r-2),f)=1\\ (n,f^2) \mid 4k - r^2  
	}\frac {S(f,\beta;|c_f^r|) P(f)}
	{\varphi (f^\gamma)}=F_1 F_2 F_3,
}
where
\bs{
F_1 &= \pr{q \nmid 2nr(r-2)} \Bigl(1 + \sum_{i \ge 1} \frac{S(q^i, \beta, |c_f^r|) P(q^i)}{\varphi(q^{i\gamma})}\Bigr), 
}
\bs{
F_2  &= \pr{q \mid n \\q \nmid 2r(r-2)\\ v_q(n) > \nu_r}
\Bigl(1 + \su{j \ge 1\\ (n,q^{2j}) \mid 4k - r^2} \frac{S(q^j, \beta, |c_f^r|) P(q^j)}{\varphi(q^{j\gamma})}\Bigr),}
\bs{
F_3 &= \pr{q \mid n \\q \nmid 2r(r-2)\\ v_q(n) \le \nu_r}
\Bigl(1 + \su{j \ge 1\\ (n,q^{2j}) \mid 4k - r^2} \frac{S(q^j, \beta, |c_f^r|) P(q^j)}{\varphi(q^{j\gamma})}\Bigr).
}
Recalling the definitions of $P(f)$ and $S(f,\beta;F)$ in \eqref{M(f)} and using Lemma \ref{lem:cfr(m)}, 
we see for $\gamma \in [2.37,3)$, $\beta=2$ and $\beta \in [3/2, 2]$, $\gamma =3$ that
\eqs{
F_1 
< \pr{q \nmid 2nr(r-2)} \Bigl(1 + \Bigl(1 + \frac{q-1}{q(q^{2(\beta-1)}-1)}\Bigr) \frac{q}{(q-2)(q^{\gamma} -1)}  	
    \Bigr) \ll 1,
}
\eqs{
F_2  
= \pr{q \mid n \\q \nmid 2r(r-2)\\ v_q(n) > \nu_r} \Bigl(1 + \su{2j < \nu_r} \frac 1 {q^{(\gamma-2) j}} + \delta_{2 \mid \nu_r} \frac{ q^{\nu_r+\beta-1}}{q^{\nu_r\gamma/2} (q^{\beta-1} - 1)}\Bigr) < \tau(n), }
and
\mult{
F_3 
= \pr{q \mid n \\q \nmid 2r(r-2)\\ v_q(n) \le \nu_r}
\biggl(1 + \su{2j < v_q (n)} \frac 1 {q^{(\gamma-2) j}} + \frac{q^{v_q(n)+ \gamma (1 - \ceil{v_q(n)/2})}}{q^\gamma-1} \\
\cdot \Bigl( 1 + \frac{q-1}{q(q^{2(\beta-1)}-1)}\Bigr) \biggr) < \tau(n), 
}
since all terms but the first in each factor of $F_2$ and $F_3$ are less than 1 and the number of terms does not exceed $v_q(n)$ for each $q \mid n$. Inserting these bounds into \eqref{eq:F1F2F3}, and then using \eqref{costofcompletion},  we find that for any $\beta \in (3/2,2]$ and $\gamma \in [2.37, 3]$,  \eqref{rankintrick} becomes
\eqs{ 
E(U,L) \ll \pr{q \mid n} \Bigl(1- \frac 1 {q^{\beta-1}}\Bigr)^{-1} \frac {\tau(n)^2}{(2\beta-3)U^{2-\beta}}  + \pr{q \mid n} \Bigl(1- \frac 1 q \Bigr)^{-1}  \frac{\tau(n)^2}{L^{3-\gamma}}.
}
Choosing any fixed $\beta\in (3/2, 2)$ and $\gamma=2.37$, we get the estimate in \eqref{EULbound} for  sufficiently large $x$. The bound for $E_r(x)$ defined in \eqref{eq:Erx} follows by applying \eqref{EULbound} to the difference between \eqref{incompletesum} and \eqref{Cr}. 
\end{proof}

\section{Proof of Proposition \ref{prop:asym4Flr}} \label{sec:asym4Flr}


In this section, we prove Proposition \ref{prop:asym4Flr}, which may be viewed as an averaged version of Proposition \ref{prop:fixedr}. The main difficulty is to factor the average value of 
$\mathcal{C}_r$. To overcome this, we expand the Euler factors of $\mathcal{C}_r$ at primes dividing $r(r-2)$ and $r-1$. By interchanging these expansions with the sum over $r$, we obtain an almost multiplicative function.
Similar ideas appear in \cite{MR409385} and \cite{BCD}. An additional complication in our setting is that we must  decompose the $r$-sum into congruence classes modulo $n$. 
As a result, we express the  average in terms of an almost Euler product $\mathfrak{C}_{n, k}$, together with an error term arising from the truncation of the $r$-sums and the error terms in Propositions  \ref{prop:error-term-varE} and \ref{prop:fixedr}.

\begin{proof}[Proof of Proposition \ref{prop:asym4Flr}]
From \eqref{eq:Erx}, it follows that
\eqs{
\su{|r| \le R\\ r \neq 1} \cT_l (r) = \su{|r| \le R\\ r \neq 1} \cC_r + \su{|r| \le R\\ r \neq 1} E_r (x). 
}
Using \eqref{Crfinal}, we see that
\bsc{\label{Crsum}
\su{|r| \le R\\ r \neq 1} \cC_r 
&= \cC \su{|r| \le R\\ r \neq 1, 2 \nmid r \\ (n,k+1-r)=1}  \pr{q \mid r(r-2)\\ q \nmid 2n} \frac{  q^2 - 2q - 1}{q^2 - 2q -2 }
\pr{q \mid r-1\\ q \nmid 2n}  \frac{q^2-q-1 }{q^2 - 2q -2} 
\\
& \qquad \cdot
\pr{q \mid n\\q \nmid 2(r^2-4k)} 
\Bigl(1- \frac{ \eta_r (q) } {q} \Bigr)^{-1} \pr{q \mid (n,r^2-4k) \\ q \nmid 2r(r-2)}  Q_r (q), }
where
\eqn{\label{eq:C}
\cC := \frac {4} {3\varphi (n)}  
\prod_{q \neq 2} \frac{q^2(q^2 - 2q -2)}{(q+1)(q-1)^3} 
\pr{q \mid n \\ q \neq 2}
\frac{(q-1)(q^2-1)}{q(q^2 - 2q -2)}.
}
By Proposition \ref{prop:Erx}, we have that
\eqn{\label{Ersum} 
\su{|r| \le R\\ r \neq 1} E_r (x) \ll \frac{\tau(n)^2 \log\log n}{\varphi (n) L^{0.63}} \su{1 < r \le R} \pr{q \mid (r-1)n\\ q \neq 2} \frac{q-1}{q-2}.
}

Recall that for $p\mid n$ and $r\equiv u \mod n$, we have $Q_r (p) = Q_u (p)$. Therefore, we can rewrite \eqref{Crsum} as
\eqs{
\cC
\su{u \mod n\\ (n,k+1-u)=1} \pr{q \mid n\\q \neq 2} 
\Bigl(1- \frac{\eta_u (q) } {q} \Bigr)^{-1} 
\pr{q \mid n \\  q \nmid 2u(u-2)}  Q_u (q) 
\su{|r| \le R\\ r \equiv u \mod n \\
r \neq 1, 2 \nmid r} 
\sum_{d \in \cG(r)} h_r(d),
}
where $\cG(r)$ is the set of positive square-free integers $d$ composed of odd prime divisors of $r(r-1)(r-2)$ coprime to $n$, and
\eqs{
h_r(1)=1, \quad h_r(d) = \prod_{q \mid d} h_r (q),
}
where
\eqs{
h_r (q) = 
\begin{dcases}
H_1(q) := \frac{q^2-q-1 }{q^2 - 2q -2} -1 & \text{if } q \mid r-1, \ q \nmid 2n, \\ 
H_2(q) := \frac{  q^2 - 2q - 1}{q^2 - 2q -2 } -1 & \text{if } q \mid r(r-2), \ q \nmid n, \\
0 & \text{otherwise.}
\end{dcases}
}
Let $\cD(R)$ be the set of positive square-free integers $d$ with $(d,n)=1$ composed of odd prime divisors of $r(r-1)(r-2)$ for some odd $r \neq 1$ with $r \equiv u \mod n$ and $|r| \le R$. Then, we have
\multn{\label{eq:r-sum-multiplicative}
\su{|r| \le R\\ r \equiv u \mod n \\
	r \neq 1, 2 \nmid r} 
\sum_{d \in \cG(r)} h_r(d)
 =
\su{d \in \cD(R)} \su{|s| \le R\\ s \equiv u \mod n \\ s \neq 1, 2 \nmid s }  h_s(d), \\
= \su{d \in \cD(R)} \sum_{f_d} \prod_{q \mid d}  H_{f_d (q)} (q) \su{|s| \le R\\ s \equiv u \mod n \\ s \neq 1, 2 \nmid s \\ 
\forall q\mid d, H_{f_d (q)} (q) =  h_s(q) }
1,}
where the middle sum runs over all functions $f_d :\{ q: q \mid d\} \to \{ 1,2 \}$.
Note that $f_d$ is completely  determined by the vector  \[
(F_q)_{q\mid d}, \quad F_q
    \in \{1, 2\}.
\]
For each fixed $f= f_d$, the  $s$-sum above equals
\eqs{
 \Bigl(\frac{2R+1}{d[2,n]} + O(1)\Bigr) \prod_{q \mid d} f_d(q) = \frac{(2,n)}{dn} R \prod_{q \mid d} f_d(q) + O \bigl( 2^{\omega(d)} \bigr).
}
Inserting the above into \eqref{eq:r-sum-multiplicative}, we obtain
\eqs{
\su{|r| \le R\\ r \equiv u \mod n \\
	r \neq 1, 2 \nmid r} 
\sum_{d \in \cG(r)} h_r(d) 
= R \frac{(2,n)}{n} \su{d \in \cD(R)} G(d) + O \bigl( \mathcal{E}(R)\bigr), 
} 
where 
\bs{
 \mathcal{E}(R) &:=    \su{d \in \cD(R)} 2^{\omega(d)} \sum_{f_d} \prod_{q \mid d}  H_{f_d (q)} (q), \\
 G(d) & :=\frac{1}{d}\sum_{f_d}\prod_{q\mid d}H_{f_d(q)}(q)f_d(q).
}
We now show that $G(d)$ is multiplicative. 
For $(d_1, d_2)=1$, we have
\mult{
\sum_{\substack{(F_q)_{q\mid d_1d_2}\\ F_q\in \{1,2\}}} \prod_{q\mid d_1d_2}H_{F_q}(q) F_q=\sum_{\substack{(F_q)_{q \mid d_1}\\ F_q \in \{1,2\}}} \sum_{\substack{(F_q)_{q\mid d_2}\\ F_q \in \{1,2\}}}\prod_{q\mid d_1} H_{F_q}(q) F_q \cdot \prod_{q\mid d_2} H_{F_q}(q) F_q, \\
=\sum_{\substack{(F_q)_{q \mid d_1}\\ F_q \in \{1,2\}}} \prod_{q\mid d_1} H_{F_q}(q) F_q  \cdot \sum_{\substack{(F_q)_{q\mid d_2}\\ F_q \in \{1,2\}}}  \prod_{q\mid d_2} H_{F_q}(q) F_q.
}
Therefore, 
\eqs{
\sum_{f_{d_1d_2}}=\sum_{f_{d_1}}\sum_{f_{d_2}}, \; \text{ if $(d_1, d_2)=1$},
}
and we obtain
\bs{
R \frac{(2,n)}{n} \su{d \in \cD(R)} G(d) 
&= R \frac{(2,n)}{n} \pr{2 < q \le R\\ q \nmid n} 
\Bigl(1 +  \frac{ 1 }{q} \sum_{f_q} H_{f_q (q)} (q)  f_q(q) \Bigr),
\\ 
&= 
R \frac{(2,n)}{n} \pr{2 < q \le R\\ q \nmid n} 
\Bigl(1 +  \frac{   H_1(q) + 2H_2 (q) }{q} \Bigr), 
%
\\
&= R \frac{(2,n)}{n} \pr{q \nmid 2n} 
\frac{q^3 - 2 q^2 - q + 3}{q (q^2 - 2 q - 2)}
+ O(1/n).
}
Similarly, we can show that $\cE(R)$ factors as
\bs{
\mathcal{E}(R)
&= \pr{2 < q \le R\\ q \nmid n} 
\Bigl(1 +  2H_1(q) + 4H_2 (q) \Bigr), \\
&=
\pr{2 < q \le R\\ q \nmid n} 
\Bigl( 1 + \frac{2 (q + 3)} {q^2 - 2 q - 2} \Bigr)
\ll (\log R)^2,
}
With these computations, we obtain
\eqn{\label{Crsumfinal}
\su{|r| \le R\\ r \neq 1} \cC_r  
= \ckn R  +  O \bigl( \cC \cS(n; k) \log^2R \bigr) ,
}
where
\eqn{\label{def:Snk}
\cS(n; k) := \su{u \mod n\\ (n,k+1-u)=1} \pr{q \mid n\\q \neq 2} 
\Bigl(1- \frac{\eta_u (q) } {q} \Bigr)^{-1} 
\pr{q \mid n \\  q \nmid 2u(u-2)}  Q_u (q),
}
and
\eqn{\label{def:cnk}
\ckn := \cC \cS(n; k)  \frac{(2,n)}{n} \pr{q \nmid 2n} \frac{q^3 - 2 q^2 - q + 3}{q (q^2 - 2 q - 2)}.
}
After simplifying, we get

\eqn{\label{cnksimpler}
\ckn = \frac {4(2,n)} {3\varphi (n)} \frac{\cS(n; k)}{n} 
\prod_{q \nmid 2n} \frac{q(q^3 - 2 q^2 - q + 3)}{(q+1)(q-1)^3}   
\pr{q \mid n\\q \neq 2} \frac{q}{q-1}.
}
Combining \eqref{Ersum} and \eqref{Crsumfinal}, we find that 
\mult{
\su{|r| \le R\\ r \neq 1} \cT_l (r) - \ckn R  \\
\ll  \cC \cS(n; k) \log^2R 
+ \frac{\tau(n)^2 \log\log n}{\varphi (n) L^{0.63}} \su{1 < r \le R} \pr{q \mid (r-1)n\\ q \neq 2} \frac{q-1}{q-2}.
}

For the claimed error term, we apply the bound for
$\cC \cS(n; k)$ from Lemma \ref{lem:CSnkbound} to \eqref{Crsumfinal}, and then combine it with the error terms from Propositions \ref{prop:fixedr} and \ref{prop:error-term-varE}. This completes the proof of Proposition \ref{prop:asym4Flr}.

\end{proof}

\section{The Euler Product Formula of $\ckn$}\label{sec:fullEuler-product}

The main result of this section is the Euler product formula for $\ckn$ given in Proposition \ref{prop:Euler-product}. To show this result, we begin by showing that $\cS(n;k)$ is multiplicative, reducing the evaluation to prime powers. For each prime power $p^v \| n$, we derive explicit formulas for $\cS(p^v;k)/p^v$, considering separately the cases where $\bigl(\tfrac{k}{p}\bigr)=\pm1$. Combining these local computations will then yield the desired Euler product for $\ckn$, as well as an upper bound for $\mathcal{C}\cS(n;k)$ given in Lemma \ref{lem:CSnkbound}.



\begin{lemma}\label{lem:multi-Sp}
We have 
\eqs{\cS(n; k) = \pr{p^v \| n} S(p^v;k).}
\end{lemma}
\begin{proof}
Clearly 
\eqs{
s_u(n):=\pr{q \mid n\\q \neq 2} 
\Bigl(1- \frac{ \eta_u(q) } {q} \Bigr)^{-1} \pr{q \mid n \\  q \nmid 2u(u-2)}  Q_u (q)
}
is multiplicative in $n$, and $s_u(n)=s_{u'}(n)$ if $u\equiv u'\mod n$. If $(n_1, n_2)=1$, then by the Chinese remainder theorem, 

\eqs{\sum_{\substack{u\pmod {n_1n_2}\\ (n_1n_2, k+1-u)}} s_u(n_1n_2) =\sum_{\substack{u_1\pmod {n_1}\\ u_2\pmod {n_2}\\ (n_1n_2, k+1-(n_2\bar{u}_2u_1+n_1\bar{u}_1 u_2))}} s_{u_1}(n_1)s_{u_2}(n_2)}

\eqs{
\quad \quad \quad \quad \quad \quad \quad \quad \quad   =\sum_{\substack{u_1\pmod {n_1}\\ u_2\pmod {n_2}\\ (n_1, k+1-u_1)}} s_{u_1}(n_1)\cdot \sum_{\substack{u_1\pmod {n_1}\\ u_2\pmod {n_2}\\ (n_2, k+1-u_2)}} s_{u_2}(n_2),}
where $n_2\bar{u}_2\equiv 1\pmod {n_1}$ and $ n_1\bar{u}_1\equiv 1\pmod {n_2}$. Equivalently, 
$\cS(n_1n_2; k)=\cS(n_1; k) \cS(n_2; k)$. This shows the multiplicativity of the function $\cS(n; k)$.
\end{proof}

Next, we look at the local factors. Note that
\eqs{
\cS(p^v;k) = \su{u \mod p^v\\ (p,k+1-u)=1} 
\pr{q \mid p\\q \neq 2} 
\Bigl(1- \frac{\eta_u (q) } {q} \Bigr)^{-1} 
\pr{q \mid n \\  q \nmid 2u(u-2)}  Q_u (q).
}
Thus, $\cS(2^{v_2(n)}; k) = 2^{v_2(n)-1}$ if $2 \mid n$, and for odd $p \mid n$,
\eqs{
\cS(p^{v_p(n)}; k)  = \su{u \mod p^{v_p(n)}\\ p \nmid k+1-u} 
\Bigl(1- \frac{ \eta_u (p) } {p} \Bigr)^{-1} 
\pr{q \mid p \\ q \nmid  u(u-2) }  Q_u (q).
}
From now on, we fix an odd prime $p\mid n$ and let $v=v_p(n)$. 

Using the definitions of $\cS(p^v; k), Q_u(p)$ and $\eta_u(p)$, we split the local factor $\cS(p^v; k)$ as
\eqs{
\cS(p^v; k) = \cS_{\mathrm{reg}} (p^v;k) + \cS_{\mathrm{root}} (p^v; k),
}
where 
\bs{
\cS_{\mathrm{reg}} (p^v;k)
	&=
	\su{u \mod p^{v} \\ p\nmid k+1-u \\ p\nmid u^2-4k}
	\p{1-\frac{\eta_u (p)} p }^{-1}, \\
\cS_{\mathrm{root}} (p^v; k) 
	&=
	\su{u \mod p^{v} \\ p \mid u^2 - 4k\\ p \nmid k+1-u} \pr{q \mid p \\ q \nmid  u(u-2) }  Q_u (q).
}
Since the Legendre symbol depends only on $u \mod p$, each residue
$a \mod p$ has exactly $p^{v-1}$ lifts modulo $p^{v}$. Thus
\begin{equation}
    \label{eq:Sreg}
	\cS_{\mathrm{reg}} (p^v;k)
	=
	p^{v-1}
	\su{a \mod p \\ a^2 \not\equiv 4k \pmod p \\ a \not\equiv k+1 \pmod p}
	\p{1-\frac{\eta_a(p)} p }^{-1}.
\end{equation}

\subsection{Evaluation of local factors when $(k/p)=-1$}

\begin{proposition} \label{prop:nonresidueSpn}
Assume that $p^v \| n$ and $(k/p)=-1$. Then, 
\eqs{
\frac{\cS(p^v; k)}{p^v} =
\frac{p-2}{p-1}.}
\end{proposition}

\begin{proof}
Since $\bigl(\frac{k}{p}\bigr) = -1$, the congruence $a^2 \equiv 4k \pmod p$ has no solutions. Consequently, $\cS(p^v; k) = \cS_{\mathrm{reg}}(p^v; k)$ and $\eta_a(p) \in \{\pm 1\}$ for every $a \bmod p$. Moreover, since $p \nmid 4k$, standard character sum evaluations\footnote{We briefly sketch a proof of \eqref{eq:quadratic}: Let 
\[
W:=\su{a \mod p} \left(\kronecker{a^2-4k}{p}+1\right), 
\]
which is the number of solutions $(x, y)\in \F_p\times \F_p$ of the equation $4k = x^2 - y^2$. Putting $u=x-y, v=x+y$ and using $p\nmid 4k$, it follows that $W=p-1$, which leads to the claim. } yield
\eqn{\label{eq:quadratic}
\su{a \bmod p \\ a^2 \not\equiv 4k \pmod p} \eta_a(p) 
= \sum_{a \bmod p} \eta_a(p) 
= -1.
}
Let 
\eqs{N_+ := \sum_{\substack{a \bmod p \\ \eta_a(p) = 1}} 1, \quad 
N_- := \sum_{\substack{a \bmod p \\ \eta_a(p) = -1}} 1.}
Combining the relations
\begin{align*}
N_+ - N_- = -1 \quad \text{and} \quad N_+ + N_- = p,
\end{align*}
we obtain
\[
N_+ = \frac{p-1}{2} \quad \text{and} \quad N_- = \frac{p+1}{2}.
\]

Since $\eta_{k+1}(p) = 1$, removing this single term from $N_+$ and applying \eqref{eq:Sreg}, we find
\begin{align*}
\frac{\cS(p^v; k)}{p^{v - 1}} 
&= \sum_{\substack{a \bmod p \\ a^2 \not\equiv 4k \pmod p \\ a \not\equiv k+1 \pmod p \\ \eta_a(p) = 1}} \left( 1 - \frac{1}{p} \right)^{-1} 
+ \sum_{\substack{a \bmod p \\ a^2 \not\equiv 4k \pmod p \\ a \not\equiv k+1 \pmod p \\ \eta_a(p) = -1}} \left( 1 + \frac{1}{p} \right)^{-1}, \\[1.5ex]
&= \left( \frac{p - 1}{2} - 1 \right) \frac{p}{p - 1} + \left( \frac{p + 1}{2} \right) \frac{p}{p + 1} = \frac{p(p - 2)}{p - 1}. 
\end{align*}
\end{proof}

\subsection{Evaluation of local factors when $(k/p)=1$}

Define 
\begin{equation}\label{eq:ak}
    a_{k}=a_k(p)= 1 + \eta_{k+1} (p) = \begin{cases}
        1 & \text{if } k\equiv 1 \mod p,\\
        2  & \text{if } k\neq 1 \mod p.
    \end{cases}
\end{equation}
\begin{lemma}\label{lem:Sp-reg}
Assume that $p^v \| n$ and $(k/p)=1$. Then, 
\eqs{\label{sreg}
\frac{\cS_{\mathrm{reg}} (p^v;k)}{p^v} = \frac p {p+1} -  \frac {a_k} {p-1}.}
\end{lemma}
\begin{proof}
Since $\bigl( \frac k p \bigr)=1$, the congruence $a^2\equiv4k \pmod p$ has exactly two solutions, so
\eqs{
	N_0 = \#\set{a \mod p : \eta_a(p)=0}=2.
}
Moreover, since $p\nmid 4k$, we can use \eqref{eq:quadratic} to obtain that
\eqs{
	N_+ := \su{a \mod p\\ \eta_a(p) = 1} 1 = \frac{p-3}{2}, \qquad N_- := \su{a \mod p\\ \eta_a(p) = -1} 1 = \frac{p-1}{2}.
}
Now we remove the residue $a\equiv k+1 \pmod p$ from $N_+$. Since
\eqs{
	(k+1)^2-4k=(k-1)^2,
}
we have
\eqs{
\eta_{k+1} (p) =
	\ar{
		1 & k \not\equiv 1 \pmod p, \\
		0 & k \equiv 1 \pmod p .
	}
}
Hence, 
\eqs{
\cS_{\mathrm{reg}} (p^v;k) 
=
	(N_+- \eta_{k+1}(p)) \frac{p}{p-1}
	+
	N_-\frac{p}{p+1} ,
}
which gives the desired result. 
\end{proof}
\begin{lemma}\label{lem:Sp-root}
Assume that $p^v \| n$ and $(k/p)=1$. Then, 
\eqs{
\frac{\cS_{\mathrm{root}} (p^v; k) }{p^v} 
= a_k \frac{p}{p^2-1},
}
where $a_k$ is defined in \eqref{eq:ak}.
 \end{lemma}
\begin{proof}
Write
\eqn{\label{Vtsum}
\cS_{\mathrm{root}} (p^v; k) 
= \sum_{1 \le t \le v} V_t (p),
}
where
\eqs{
V_t (p) = \su{u \mod p^{v}\\ v_p (u^2-4k)=t \\ p \nmid k+1-u} 
\pr{q \mid p \\ q \nmid  u-2 }  Q_u (q) .
}
For each $1\le t\le v$, residue classes $u$ counted by $V_t(p)$ are of the form $u = \pm u_0 + p^t w$, where $\pm u_0$ are the two roots of $u^2 -4k$ modulo $p$ and $w\in (\Z/p^{v-t}\Z)^\times $. Therefore,
\eqs{
V_t (p) = \su{w \mod p^{v-t}\\ p \nmid w\\ p \nmid k+1\mp u_0 } \pr{q \mid p \\ q \nmid \pm u_0-2 }  Q_{\pm u_0 + p^tw} (q).
}
\textbf{Case 1.} First, assume $k=1 \mod p$ and $t < v$. Clearly, $u_0\equiv \pm 2\pmod p$. In this case, $u_0=-2$ and
\[
0\not\equiv k+1-u_0 \equiv 4\pmod p,
\]
since $p$ is an odd prime. Therefore, by \eqref{Qr} we have
\bs{
V_t (p) 
&= \su{w \mod {p^{v-t}}\\ p\nmid w} 	
\Bigl( 
\frac{p(p^{\ceil{t/2}}-1)}
{p^{\ceil{t/2}} (p-1)} + \frac{\delta_{2 \mid t } p }{p^{t/2} (p - \psi_{-2 + p^t w} (p))} \Bigr),
\\	
&=	\varphi (p^{v-t}) \frac{p(p^{\ceil{t/2}}-1)}
{p^{\ceil{t/2}} (p-1)} + 
\frac{\delta_{2 \mid t } p }{p^{t/2}} \su{w \mod {p^{v-t}}\\ p\nmid w} 	\frac 1 {p - \bigl( \frac{ -4 w}{p} \bigr) }, 
}
where 
\[
\delta_{2\mid t}=\begin{cases}
    1 & \text{ if $2\mid t$},\\ 
    0 & \text{ if $2\nmid t$},
\end{cases}
\]
since
\eqs{
	\psi_{-2 + p^t w} (p) = \Bigl( \frac{ ( (-2 + p^t w)^2-4)/p^t} p \Bigr) = \Bigl( \frac{ -4 w}{p} \Bigr).
}
As $w$ ranges over the units modulo $p^{v-t}$, $\psi$ will take both values $\pm 1$ equally often, so we get 
\eqn{\label{Vtsmallervaluation}
V_t (p) = \varphi (p^{v-t}) \Bigl(\frac{p(p^{\ceil{t/2}}-1)}
{p^{\ceil{t/2}} (p-1)} + 
 \frac{\delta_{2 \mid t } p }{p^{t/2}} \frac p{p^2-1}\Bigr).
}
Using \eqref{Qr} again, we find
\eqn{\label{Vtequalvaluation}
V_v (p) = 
\frac{p(p^{\ceil{v/2}}-1)}{p^{\ceil{v/2}} (p-1)} + \frac{p^{2+v-3\ceil{v/2}}}{p^2-1}.
}
\textbf{Case 2.} Assume $k\not\equiv  1 \mod p$. Then, 
\[
0\not \equiv k+1\pm u_0\pmod p
\]
for $u_0$ satisfying $u_0^2\equiv 4k\pmod p$. Therefore, we shall get twice the same values we get when $k=1 \mod p$.

Inserting \eqref{Vtsmallervaluation} and \eqref{Vtequalvaluation} into \eqref{Vtsum}, we find that
\bs{
\cS_{\mathrm{root}} (p^v; k) 
&= a_k \sum_{1 \le t < v} \varphi (p^{v-t}) \Bigl(\frac{p(p^{\ceil{t/2}}-1)}
{p^{\ceil{t/2}} (p-1)} + 
\frac{\delta_{2 \mid t } p }{p^{t/2}} \frac p{p^2-1}\Bigr) 
\\
& \quad + a_k \Bigl(\frac{p(p^{\ceil{v/2}}-1)}{p^{\ceil{v/2}} (p-1)} + \frac{p^{2+v-3\ceil{v/2}}}{p^2-1}\Bigr). }

If $v=2m$, $\cS_{\mathrm{root}} (p^v; k)/a_k$ becomes 
\eqs{
\frac{p(p^m-1)} {p^m(p-1)} + \frac{p^{2-m}}{p^2-1}  +
\sum_{1 \le t < 2m} p^{2m-t}  \Bigl(\frac{p^{\ceil{t/2}}-1}
{p^{\ceil{t/2}} } + 
\frac{\delta_{2 \mid t }  }{p^{t/2}} \frac p{p  +1}\Bigr).}
Separating even and odd values of $t$, this simplifies to
\eqs{
\frac{p^m-1}{p^m} \Bigl( \frac p {p-1} + p  \Bigr)   + \frac{p^{2-m}}{p^2-1}
+ \sum_{1 \le t < m}  p^{2m-3t} \Bigl( p^t (p+1) +
\frac {p - (p+1)^2} {p+1}   
\Bigr)
}
\mult{=\frac{-p^{3-m}+p^3+p^2}{p^2-1} +  (p+1) \sum_{1 \le t < m}  p^{2m-2t}  - \frac{p^3-1}{p^2-1} \sum_{1 \le t < m}  p^{2m-3t}, \\
= \frac{-p^{3-m}+p^3+p^2 + (p+1) (p^{2m}- p^2)- (p^{2m}- p^{3-m})} {p^2 -1 } = \frac{p^{2m+1} } {p^2 -1 }.
}
A similar computation will give the same result when $v$ is odd. Thus, the claimed result follows.

\end{proof}

\begin{proposition}\label{prop:Sp-formula}
Assume that $p^v \| n$ and $(k/p)=1$. Then, 
\eqs{
\frac{\cS(p^v; k)}{p^{v}} = \begin{dcases}
    \dfrac{p^2-p-1}{p^2 - 1} & \text{if } k=1 \mod p,\\
    \dfrac{p-2}{p - 1} & \text{if } k\neq 1 \mod p.
\end{dcases}
}
\end{proposition}
\begin{proof}
Combining Lemmas \ref{lem:Sp-reg} and \ref{lem:Sp-root},  we obtain
\eqs{
\frac{\cS(p^v; k)}{p^{v}} = \frac p {p+1} -  \frac {a_k} {p-1}+a_k \frac{p}{p^2-1} = \frac{p}{p+1} - \frac{a_k}{p^2-1},
}
which leads to the desired result.
\end{proof}

Using this proposition, we immediately get the following upper bound, which was used in the proof of Proposition \ref{prop:asym4Flr}.

\begin{lemma}\label{lem:CSnkbound}
Recall the definition \eqref{eq:C}. We have 
\eqs{
\cC \cS (n; k) \ll (\log \log n)^2.
}
\end{lemma}
\begin{proof}
This follows immediately by the definition \eqref{eq:C}, Lemma \ref{lem:multi-Sp}, Proposition \ref{prop:nonresidueSpn}, Proposition \ref{prop:Sp-formula}, and the fact that
\eqs{
\frac{2q+1}{q^2 - 2q -2} = \frac 2 q + O (q^{-2}). 
}
\end{proof}

We are now ready to state and prove the main result of this section. 
\begin{proposition} \label{prop:Euler-product}
The constant $\ckn$ given in \eqref{cnksimpler} is equal to $2C^{\Prime}_{n,k}$, where $C^{\Prime}_{n,k}$ is given in \eqref{eq:koblitzconstant}.
\end{proposition}
\begin{proof}
 The Euler product formula follows from \eqref{cnksimpler} together with Lemma \ref{lem:multi-Sp}, Proposition \ref{prop:nonresidueSpn} and Proposition \ref{prop:Sp-formula}.
\end{proof}

\begin{appendix}

\section{Proof of Lemma \ref{lem:average_estimate}} \label{app:lem-proof}


The following proof is a slightly less detailed version of the arguments described in \cite[Section 3.1]{BCD}.

\begin{proof}[Proof of Lemma \ref{lem:average_estimate}]
First, note that we can switch the sums: 
\mult{
\frac{1}{|\sF|} \su{E \in \sF} \pi_E^{\Prime}(x; k, n) \\
=  \frac 1 {|\sF|} \su{p \le x		\\ p \equiv k \mod n} 	\# \{ E_{a, b}  \in \sF : p \nmid \Delta_{E_{a, b}}, |E_{a, b} (\F_p)| \text{ prime} \}=\Sigma_1 + \Sigma_2,}
where
\bs{
\Sigma_1 &= \frac 1 {|\sF|} \su{p \le x \\ p \equiv k \mod n} \# \{ E_{a,b}  \in \sF : p \nmid ab\Delta_{E_{a, b}}, \ |E_{a, b} (\F_p)| \text{ prime} \}, \\
\Sigma_2 &= \frac 1 {|\sF|} \su{p \le x \\ p \equiv k \mod n} \# \{ E_{a,b}  \in \sF : p \mid ab, \ p \nmid \Delta_{E_{a, b}}, \ |E_{a, b} (\F_p)| \text{ prime} \}
.
}
For $n < (\log x)^C$, 
\bsc{\label{sigma2bound}
\Sigma_2 \ll & \frac{1}{|\sF|}\sum_{\substack{p\le x\\ p\equiv k\pmod n}}\sum_{\substack{|a|\le A, |b|\le B\\ ab\equiv 0\pmod p}} 1  \\
&\frac{1}{|\sF|}\su{p\le x\\ p\equiv k\pmod n}
\Bigl( 
    B \Bigl( \frac A p + O(1) \Bigr) + 
    A \Bigl( \frac B p + O(1) \Bigr)
\Bigr), \\
\ll & \frac{1}{\varphi(n)}\log \log x +  \frac{ (A+B)x  }{|\sF|\varphi(n) \log x}  \ll \frac x {\varphi(n) \log^3 x },
} 
provided that $A, B > x^\eps$. 

The rest of the curves $E \in \sF$ in $\Sigma_1$ will be counted according to their $\F_p$ isomorphism classes as follows. Assume $p\neq 2, 3$. Given $s,t \in \F_p$ and
$E_{s, t}$ defined over $\F_p$ and $ E_{a,b} \in \sF$, we have $E_{a,b} \pmod p \simeq E_{s,t}$  if and only if $a \equiv su^4 \mod p$ and $b \equiv tu^6 \mod p$ for some $u \in \F_p^\times$. Therefore, 
\eqs{
\Sigma_1= \frac 2 {|\sF|} \su{p \le x\\ p \equiv k \mod n} \frac 1 {p-1} \su{(s,t) \in \F_p^2\\ p \nmid st\Delta_{s,t}\\|E_{s,t} (\F_p)| \text{ prime}} \# \{ E \in \sF: E_p \simeq E_{s,t} \}
}
since $|\text{Aut} (E_{s,t})| = 2$ when $p \nmid st$. Using Dirichlet characters $\chi_0, \chi, \chi_1, \chi_2$  mod $p$ with $\chi_0$ being the principal character,  we can write 

\bsc{\label{eq:count-EC}
&\# \{ E \in \sF : E_p \simeq E_{s,t}  \} \\
&= \frac{1}{2} \sum_{u \in \F_p^\times} \sum_{|a| \le A} \sum_{|b| \le B} \biggl(\frac{1}{p-1} \sum_{\chi_1 \mod p} \chi_1(su^4) \overline{\chi}_1(a) \biggr) \\
    & \qquad \cdot \biggl(\frac{1}{p-1} \sum_{\chi_2 \mod p} \chi_2(tu^6) \overline{\chi}_2(b)\biggr), \\	
        & = \frac{1}{2(p-1)^2} \sum_{\chi_1, \chi_2 \mod p} \chi_1(s)\chi_2(t) \cA(\overline{\chi}_1) \cB(\overline{\chi}_2) \sum_{u \in \F_p^\times} (\chi_1^4 \chi_2^6) (u), \\
       & =  
	\frac 1 {2(p-1)} \sum_{\chi_1^4 \chi_2^6 = \chi_0} \chi_1(s)\chi_2(t) \cA(\overline{\chi}_1) \cB(\overline{\chi}_2),
}
where
\eqs{
	\cA(\chi) := \sum_{|a| \le A} \chi(a), \quad \cB(\chi) := \sum_{|b| \le B} \chi(b).
}
For $\chi_1 = \chi_2 = \chi_0$, we get the main term 
\bsc{\label{eq:main-term}
	\frac{1}{2(p-1)} \su{|a| \le A \\ p \nmid a} 1  \su{|b| \le B \\ p \nmid b} 1 
	& = \frac{|\sF|}{2(p-1)} + O  \Bigl(\frac{AB}{p^2} + \frac{A}{p} + \frac{B}{p}\Bigr),
}
while the contribution when exactly one of $\chi_1 = \chi_0$ or $\chi_2 = \chi_0$ holds is
\eqn{\label{eq:non-trivial-car}
	\ll \frac{A}{p} \su{\chi_2 \neq \chi_0 \\ \chi_2^6 = \chi_0} \abs{\cB(\chi_2)} + 
	\frac{B}{p} \su{\chi_1 \neq \chi_0 \\ \chi_1^4 = \chi_0} \abs{\cA(\chi_1)}.
}
These estimates are independent of $s$ and $t$. Using the expression \eqref{eq:count-EC} and the asymptotic bounds \eqref{eq:main-term} and \eqref{eq:non-trivial-car}, we find that 
\bs{&\Sigma_1 - 
\su{p \le x\\ p \equiv k \mod n} \frac 1 {(p-1)^2} \su{(s,t) \in \F_p^2\\ p \nmid st\Delta_{s,t}\\|E_{s,t} (\F_p)| \text{ prime}} 1 \\
& \ll 
	\frac 1 {|\sF|} \su{p \le x\\ p \equiv k \mod n} \frac 1 { (p-1)^2} 
	\Bigg|
	\su{\chi_1^4 \chi_2^6 = \chi_0 \\\chi_1, \chi_2 \neq \chi_0}  \cA(\overline{\chi}_1) \cB(\overline{\chi}_2)
	\su{(s,t) \in \F_p^2\\ p \nmid st\Delta_{s,t}\\|E_{s,t} (\F_p)| \text{ prime}} 
	\chi_1(s)\chi_2(t)
	\Bigg| \nonumber 
	\\
&	+ \frac 1 {|\sF|} \su{p \le x\\ p \equiv k \mod n}   
	\Bigl(\frac{AB}{p  } + A+ B + A \su{\chi_2 \neq \chi_0 \\ \chi_2^6 = \chi_0} \abs{\cB(\chi_2)} + 
	B \su{\chi_1 \neq \chi_0 \\ \chi_1^4 = \chi_0} \abs{\cA(\chi_1)} \Bigr).  }

Using triangle and H\"older inequalities, for any $l>0$, 
\bs{
&\su{p \le x	\\ p \equiv k \mod n} \su{\chi_2 \neq 1 \\ \chi_2^6 = 1} \abs{\cB(\chi_2)} 
\le 2 \su{p \le x	\\ p \equiv k \mod n} \su{\chi_2 \neq 1 \\ \chi_2^6 = 1} \bigg|\sum_{b \le B} \chi_2 (b)\bigg| \\
& \le 2 
\biggl(\su{p \le x	\\ p \equiv k \mod n} \su{\chi_2 \neq 1 \\ \chi_2^6 = 1} 1 \biggr)^{1-\frac 1{2l}}
\biggl( \su{p \le x	\\ p \equiv k \mod n} \su{\chi_2 \neq 1 \\ \chi_2^6 = 1} \bigg|\sum_{b \le B} \chi_2 (b) \bigg|^{2l} 
\biggr)^{\frac 1 {2l}}, \\
&\ll_l 
\Bigl(\frac{x}{\varphi(n)\log  x}\Bigr)^{1-\frac 1{2l}}
\biggl(\su{p \le x	\\ p \equiv k \mod n} \su{\chi_2 \neq 1 \\ \chi_2^6 = 1} \bigg|\sum_{b \le B^l} \tau_l (b) \chi_2 (b) \bigg|^2 
\biggr)^{\frac 1 {2l}},
}
where $\tau_l(b)=\sum_{\substack{b_1\ldots b_l=b\\ 1\le b_1, \ldots, b_l\le B}}1 $.

By Gallagher's large sieve inequality, which states that
\eqs{
	\sum_{q \le X} \frac{q}{\varphi(q)} \sideset{}{^*}\sum_{\chi} \left| \sum_{n = M+1}^{M+N} a_n \chi(n) \right|^2 
	\ll (X^2 + N) \cdot \sum_{n = M+1}^{M+N} |a_n|^2,
}
where $\sum^*$ indicates that the sum runs over primitive characters, 
we get
\mult{
\su{p \le x	\\ p \equiv k \mod n} \su{\chi_2 \neq 1 \\ \chi_2^6 = 1} \bigg|\sum_{b \le B} \chi_2 (b) \bigg|^{2l} 
\le \sum_{q \le x} \frac{q}{\varphi(q)} \sideset{}{^*}\sum_{\chi \mod q} \left| \sum_{b \le B^{l}} \tau_l (b) \chi(b) \right|^2 
\\
\ll (x^2 + B^l) \sum_{b \le B^l} \tau_l (b)^2 \ll (x^2 + B^l) B^l (\log B^l)^{l^2 -1}.
}
We can assume $B \le x$ since $\chi_2$ is periodic with period $p \le x$. Therefore,
\multn{\label{sigma1error1}
	A\su{p \le x	\\ p \equiv k \mod n} \su{\chi_2 \neq \chi_0 \\ \chi_2^6 = \chi_0} \abs{\cB(\chi_2)} 
	\\
\ll_l 
	A\Bigl(\frac{x}{\varphi(n)\log (x/n)}\Bigr)^{1-\frac 1{2l}} (x^{1/l} B^{1/2} + B) ( \log x)^{l/2 -1/2l}.
}
A similar bound with $A$ and $B$ swapped holds for 
\eqs{
	B \su{p \le x	\\ p \equiv k \mod n} \su{\chi_1 \neq \chi_0 \\ \chi_1^4 = \chi_0} \abs{\cA(\chi_1)}. 
}
Next, we apply H\"older's inequality to get

\mult{
	\Bigg|\su{\chi_1^4 \chi_2^6 = \chi_0 \\ \chi_1, \chi_2 \neq \chi_0} \cA(\overline{\chi}_1) \cB(\overline{\chi}_2) \su{(s,t) \in \F_p^2\\ p \nmid st\Delta_{s,t}\\|E_{s,t}(\F_p)| \text{ prime}} \chi_1(s)\chi_2(t)  \Bigg|^4 \\
	\hspace*{-1cm} \le 
	\su{\chi_1^4 \chi_2^6 = \chi_0 \\ \chi_1, \chi_2 \neq \chi_0} \big|\cA(\overline{\chi}_1)\big|^4 
	\su{\chi_1^4 \chi_2^6 = \chi_0 \\ \chi_1, \chi_2 \neq \chi_0}  \big|\cB(\overline{\chi}_2)\big|^4 
	\Biggl(
	\sum_{\chi_1, \chi_2 }  \bigg|\su{(s,t) \in \F_p^2\\ p \nmid st\Delta_{s,t}\\|E_{s,t}(\F_p)|  \text{ prime}} \chi_1(s)\chi_2(t)  \bigg|^2 \Biggr)^2,
}
where the conditions on $\chi_i$ are dropped on the last sum, which then becomes
\eqn{\label{eq:large-sieve}
	\su{\chi_1, \chi_2} \bigg| \su{(s,t) \in \F_p^2\\ p \nmid st\Delta_{s,t}\\|E_{s,t}(\F_p)| \text{ prime}} \chi_1(s)\chi_2(t)  \bigg|^2 = \su{(s,t) \in \F_p^2\\ p \nmid st\Delta_{s,t}\\|E_{s,t}(\F_p)| \text{ prime}} (p-1)^2 \ll (p-1)^4.
}
Furthermore, using \cite[Lemma 3]{FriIwa1985}, we have that
\eqn{\label{eq:FriIwa}
	\su{\chi_1^4 \chi_2^6 = 1 \\ \chi_1, \chi_2 \neq 1} \big|\cA(\overline{\chi}_1)\big|^4 \su{\chi_1^4 \chi_2^6 = 1 \\ \chi_1, \chi_2 \neq 1} \big|\cB(\overline{\chi}_2)\big|^4
	\ll (ABp)^2  \log^{12} p.
}
Combining \eqref{eq:large-sieve} and \eqref{eq:FriIwa}, we conclude that 
\bsc{\label{sigma1error2}
	\frac 1 {|\sF|} 
	&\su{p \le x	\\ p \equiv k \mod n} \su{(s,t) \in \F_p^2\\ p \nmid st\Delta_{s,t}\\|E_{s,t}(\F_p)| \text{ prime}}
	\frac 1 {(p-1)^2} \su{\chi_1^4 \chi_2^6 = \chi_0 \\ \chi_1, \chi_2 \neq \chi_0} \chi_1(s)\chi_2(t) \cA(\overline{\chi}_1) \cB(\overline{\chi}_2) \\
	&\ll \frac 1 {|\sF|}
	\su{p \le x	\\ p \equiv k \mod n} (ABp)^{1/2}  \log^3 p \ll \frac {\sqrt{AB}x^{3/2} \log^3 x } {\varphi(n)|\sF| \log (x/n)}.
}
By combining \eqref{sigma2bound}, \eqref{sigma1error1}, \eqref{sigma1error2}, and assuming $n\ll (\log x)^C$ and $A, B > x^\eps$, we conclude  that

\eqs{
\frac{1}{|\sF|} \su{E \in \sF} \pi_E^{\Prime}(x; k, n) - \su{p \le x\\ p \equiv k \mod n} \frac {\nu(p; \cP)}{(p-1)^2} 
}

\mult{
\ll_l \frac 1 {|\sF|} \biggl( \Bigl(\frac{x}{\varphi(n) \log(x)}\Bigr)^{1-\frac 1{2l}} ( \log x)^{\frac l 2-\frac 1 {2l}} 
\bigl(AB + x^{1/l}(A B^{1/2} + B A^{1/2}) \bigr)
\\
+ \frac {\sqrt{AB}x^{3/2} \log^3 x } {\varphi(n) \log (x)}
	\biggr) 
+ \frac x {\varphi(n) \log^3 x } ,
}
holds for sufficiently large $x$ and any $l \ge 1$, 
where the contribution from 
\[
\su{p \le x\\ p \equiv k \mod n} \frac 1 {(p-1)^2} \su{(s,t) \in \F_p^2\\ p \nmid \Delta_{s,t},\;  p\mid st \\|E_{s,t} (\F_p)| \text{ prime}} 1\ll \su{p \le x\\ p \equiv k \mod n} \frac 1 {p}\ll \frac{1}{\varphi(n)}\log\log x,
\]
is negligible. Choosing $AB > x \log^{10} x$, $l \le 3/\eps$
and then taking $x$ sufficiently large depending on $\eps$ and $C$ gives the desired bound
\eqs{O \Bigl(\frac{x}{\varphi(n) \log^3 x} \Bigr). }
\end{proof}
\end{appendix}

\bibliographystyle{amsplain}
\bibliography{biblio}

\end{document}